\documentclass[12pt]{amsart}

\usepackage[active]{srcltx}
\usepackage[all]{xy}
\usepackage{subfig}
\usepackage{hyperref}
\usepackage{mathrsfs}

\usepackage{esint}
\usepackage{amsmath}
\usepackage{amssymb,latexsym}
\usepackage{mathrsfs}
\usepackage{graphics}
\usepackage{latexsym}
\usepackage{psfrag}
\usepackage{import}
\usepackage{verbatim}
\usepackage{graphicx}
\usepackage[usenames]{color}
\usepackage{pifont,marvosym}

\theoremstyle{plain}
\newtheorem{lemma}{Lemma}[section]
\newtheorem{theorem}[lemma]{Theorem}
\newtheorem{proposition}[lemma]{Proposition}
\newtheorem{corollary}[lemma]{Corollary}

\theoremstyle{definition}

\newtheorem{definition}[lemma]{Definition}
\newtheorem{remark}[lemma]{Remark}

\numberwithin{equation}{section}

\newcommand{\R}{\mathbb{R}}
\newcommand{\N}{\mathbb{N}}
\newcommand{\Q}{\mathbb{Q}}

\newcommand{\supp}{\text{\rm supp}}
\newcommand{\loc}{\text{\rm loc}}

\newcommand{\gr}{\textrm{graph}}

\newcommand{\ve}{\varepsilon}

\newcommand{\erre}{\mathbb{R}}
\newcommand{\cI}{\mathcal{I}}

\newcommand{\f}{\varphi}

\newcommand{\T}{\mathcal{T}}
\renewcommand{\r}{\varrho}

\renewcommand{\L}{\mathcal{L}}
\newcommand{\RCD}{\mathsf{RCD}}

\newcommand{\CD}{\mathsf{CD}}
\newcommand{\Geo}{{\rm Geo}}

\newcommand{\mm}{\mathfrak m}
\newcommand{\qq}{\mathfrak q}
\newcommand{\QQ}{\mathfrak Q}

\newcommand{\Ric}{{\rm Ric}}

\newcommand{\Pe}{\mathsf{P}}
\newcommand{\sfd}{\mathsf d}
\newcommand{\Opt}{\mathrm{OptGeo}}

\begin{document}

\title[Almost euclidean Isoperimetric inequalities under local  Ricci lower bounds]
 {Almost euclidean Isoperimetric Inequalities in spaces satisfying local  Ricci curvature lower bounds} %
\author{Fabio Cavalletti}\thanks{F. Cavalletti: SISSA-ISAS, email: cavallet@sissa.it} 
\author{ Andrea Mondino} \thanks{A. Mondino: University of Warwick, Mathematics Institute, email: A.Mondino@warwick.ac.uk} 
%

%

\bibliographystyle{plain}

\begin{abstract}
Motivated by Perelman's  Pseudo Locality Theorem for  the Ricci flow, we prove that if a Riemannian manifold has Ricci curvature bounded below in a metric ball which moreover has almost maximal volume, then in a smaller ball (in a quantified sense) it holds an almost-euclidean isoperimetric inequality.
\\The result is actually established in the more general framework of non-smooth spaces satisfying local Ricci curvature lower bounds in a synthetic sense via optimal transportation.
\end{abstract}

\maketitle


\section{Introduction}

Let ${\mathbb M}^{N}_{K/(N-1)}$ be the (unique, up to isometries) simply connected manifold of constant sectional curvature equal to $K/(N-1)$; 
denote by $\textrm{Vol}_{K,N} (r)$  the volume of a metric ball of radius $r$ in  ${\mathbb M}^{N}_{K/(N-1)}$. 
\\ Given a smooth Riemannian manifold $(M,g)$ we denote with $\Ric_{g}$ the Ricci tensor, and with  ${\rm Vol}_{g}$ the  standard Riemannian volume measure.
\\The goal of this paper  is to prove the next almost euclidean isoperimetric inequality.

\begin{theorem}\label{thm:MainSmooth}
For every $K\in \R, N\in [2,\infty)\cap \N$ there exist $\bar{\varepsilon}_{K,N}, \bar{\eta}_{K,N}, \bar{\delta}_{K,N}, C_{K,N}>0$ such that the next statement is satisfied.
Let $(M, g)$ be a smooth $N$-dimensional Riemannian manifold and let $\bar{x}\in M$. Assume that $B_{1}(\bar{x})$ is relatively compact and that  for some $\varepsilon\in [0,\bar{\varepsilon}_{K,N}],  \eta \in [0,\bar{\eta}_{K,N}]$ it holds:
\begin{enumerate}
\item  ${\rm Vol}_{g}(B_{1}(\bar{x})) \geq (1-\eta) {\rm Vol}_{K,N} (1)$,
\item 
$\Ric_{g}\geq (K-\varepsilon)   g$ on  $B_{1}(\bar{x})$.
\end{enumerate}
Let $\delta \in (0, \bar{\delta}_{K,N}  ]$. Then for every finite perimeter subset $E \subset B_{\delta}(\bar{x})$ the following almost euclidean isoperimetric inequality holds:
 \begin{equation}\label{eq:AEIIsmooth}
 \Pe(E) \geq N \omega_{N}^{1/N} (1-C_{K,N} (\delta+\eta+\varepsilon) ) \; {\rm Vol}_{g}(E)^{\frac{N-1}{N}},
 \end{equation}
  where $\Pe(E)$ is the perimeter of $E$.
\end{theorem}

Let us stress that the Riemannian manifold $(M, g)$ is not requited to be complete; indeed since all our arguments are local, the assumption that $B_{1}(\bar{x})$ is relatively compact  will suffice.

\begin{remark}
\begin{itemize}
\item One of the main motivations for establishing  Theorem \ref{thm:MainSmooth} comes from Ricci flow. Indeed in the celebrated Perelman's Pseudo-Locality Theorem (see \cite[Theorem 10.1, Theorem 10.3]{Perelman}) a crucial assumption is the validity of an almost euclidean isoperimetric inequality of the form of \eqref{eq:AEIIsmooth}.  Another version of the  Pseudo-Locality Theorem, proved by Tian-Wang \cite[Proposition 3.1]{TianWang},  states that under the assumptions of Theorem  \ref{thm:MainSmooth} for $K=0$ (namely almost non-negative Ricci curvature and almost euclidean volume  on a ball) it is possible to have a  point-wise bound  on  the norm of the Riemann tensor for a definite small time along the Ricci flow around the point $\bar{x}$. For instance, such a result is used in \cite{TianWang} to prove fine structural properties of non-collapsed Gromov Hausdorff limits of almost-Einstein manifolds.   A major difficulty one encounters in proving \cite[Proposition 3.1]{TianWang} is that with more classical methods it seems hard to establish an almost euclidean isoperimetric inequality (i.e. the content of Theorem \ref{thm:MainSmooth}) and this prevents a direct application of Perelman's Pseudo-Locality Theorem; such a difficulty is bypassed by Tian-Wang via a clever technical  argument.  As mentioned in \cite[Remark 3.1]{TianWang}, the validity of the almost euclidean isoperimetric inequality proved in Theorem \ref{thm:MainSmooth} permits to have a more streamlined proof of Tian-Wang's version of the Pseudo-Locality Theorem, as one can build on top of Perelman's work \cite{Perelman} more easily. 
\item  Theorem  \ref{thm:MainSmooth} should also be compared with Colding's \cite[Theorem 0.8]{Colding}, stating roughly that a ball $B$ having almost non-negative  Ricci curvature and almost euclidean volume is Gromov-Hausdorff close to a euclidean ball; in other terms $B$ is \emph{almost euclidean in  Gromov-Hausdorff sense}. Theorem \ref{thm:MainSmooth}, specialized to the case $K=0$, states roughly that  under the same assumptions of \cite[Theorem 0.8]{Colding} the ball $B$ contains a smaller ball of definite size which is  \emph{almost euclidean in isoperimetric sense}.
\end{itemize} 
\end{remark}

Combining Theorem \ref{thm:MainSmooth} with Colding's volume convergence \cite[Theorem 0.1]{Colding} we get the next corollary. In the next statement,  $\sfd_{GH}$ denotes the Gromov-Hausdorff distance and $B_{r}^{{\mathbb M}^{N}_{K/(N-1)}}$ is a  metric ball of radius $r$ in  the model space ${\mathbb M}^{N}_{K/(N-1)}$.
\begin{corollary}
For every $K\in \R, N\in [2,\infty)\cap \N$ there exist $\bar{\varepsilon}_{K,N}, \bar{\eta}_{K,N}, \bar{\delta}_{K,N}, C_{K,N}>0$ such that the next statement is satisfied.
Let $(M, g)$ be a smooth $N$-dimensional Riemannian manifold and let $\bar{x}\in M$. Assume that $B_{1}(\bar{x})$ is relatively compact and that for some $\varepsilon\in [0,\bar{\varepsilon}_{K,N}],  \eta \in [0,\bar{\eta}_{K,N}]$ it holds:
\begin{enumerate}
\item  $\sfd_{GH}\big(B_{1}(\bar{x}), B_{1}^{{\mathbb M}^{N}_{K/(N-1)}} \big) \leq (1-\eta) $,
\item 
$\Ric_{g}\geq (K-\varepsilon)  g$ on  $B_{1}(\bar{x})$.
\end{enumerate}
Let $\delta \in (0,  \bar{\delta}_{K,N} ]$. Then for every finite perimeter subset $E \subset B_{\delta}(\bar{x})$ the almost euclidean isoperimetric inequality  \eqref{eq:AEIIsmooth} holds. 
\end{corollary}

We will actually prove the following more general result for  non-necessarily smooth metric measure spaces satisfying  Ricci curvature lower bounds in a synthetic sense.  In order to state it precisely,  let us first define the model volume function  $r\mapsto \textrm{Vol}_{K,N} (r)$ for non necessarily integer $N\in (1,\infty)$:
\begin{equation}\label{eq:defVolKN}
{\rm Vol}_{K,N} (r):= 
\begin{cases}
\displaystyle N \omega_{N} \left(\frac{N-1}{K}\right)^{\frac{N-1}{2}} \int_{0}^{\min\left\{r, \sqrt{\frac{N-1}{K}} \pi\right \}} \sin(t \sqrt{K/(N-1)})^{N-1} \, dt  , & \textrm{if}\ K>0, \crcr
\displaystyle  \omega_{N} r^{N} & \textrm{if}\  K=0 , \crcr
\displaystyle N \omega_{N}  \left(\frac{N-1}{-K}\right)^{\frac{N-1}{2}} \int_{0}^{r} \sinh(t \sqrt{-K/(N-1)})^{N-1} \, dt  , & \textrm{if}\ K<0, \crcr
\end{cases}
\end{equation}
where 
\begin{equation}\label{eq:defomegaN}
\omega_{N}:= \frac{\pi^{\frac{N}{2}}}{\Gamma\left( \frac{N}{2}+1 \right)},
\end{equation}
with $\Gamma$ denoting the Euler's Gamma function.
Clearly \eqref{eq:defVolKN} is compatible with the geometric case, in the sense that  if $N\in \N$ then ${\rm Vol}_{K,N} (r)$ is nothing but the volume of the metric ball of radius $r$ 
in ${\mathbb M}^{N}_{K/(N-1)}$. Set $\bar{r}_{K,N}:= \pi \sqrt{ (N-1)/K}$ in case $\sup_{r>0}{\rm Vol}_{K,N}(r)<1$ and  otherwise  let  $\bar{r}_{K,N}>0$ to be such that  $\textrm{Vol}_{K,N}({\bar{r}_{K,N}})=1$.

\begin{theorem}\label{thm:MainNonSmooth}
For every $K\in \R, N\in [2,\infty)$ there exist $\bar{\varepsilon}_{K,N}, \bar{\eta}_{K,N}, \bar{\delta}_{K,N}, C_{K,N}>0$ such that the next statement is satisfied.
Let $(X,\sfd,\mm)$ be a geodesic metric space  endowed with a non-negative Borel measure. For a fixed $\bar{x}\in X$, assume  that $B_{4\bar{r}_{K,N}}(\bar{x})$ is relatively compact and that $B_{4\bar{r}_{K,N}}(\bar{x})\subset \supp(\mm), \mm(B_{4\bar{r}_{K,N}}(\bar{x}))< \infty$.   
\\Assume moreover that for some $\varepsilon\in [0,\bar{\varepsilon}_{K,N}],  \eta \in [0,\bar{\eta}_{K,N}]$ it holds:
\begin{enumerate}
\item  $\mm(B_{\bar{r}_{K,N}}(\bar{x})) \geq (1-\eta) {\rm{Vol}}_{K,N}({\bar{r}_{K,N}})$, 
%
\item $\limsup_{r\downarrow 0} \frac{\mm(B_{r}(\bar{x}))}{\omega_{N} r^{N}} \leq 1+\eta$,
\item  $(X,\sfd,\mm)$ is essentially non-branching and verifies $\CD_{loc}(K-\varepsilon,N)$ inside $B_{4 \bar{r}_{K,N}}(\bar x)$.
\end{enumerate}
Then for every $\delta \in (0,\bar{\delta}_{K,N}]$ and every finite perimeter subset $E\subset B_{\delta}(\bar{x})$ the following almost euclidean isoperimetric inequality holds:
\begin{equation}\label{eq:AEII}
\Pe(E) \geq N \omega_{N}^{1/N} \Big(1-C_{K,N} (\delta+\varepsilon+\eta) \Big) \; \mm(E)^{\frac{N-1}{N}},
\end{equation}
where $\Pe(E)$ is the perimeter of $E$.
\end{theorem}

In the formulation of Theorem \ref{thm:MainNonSmooth}, point $(3)$ has to be understood as: the metric measure space $(X,\sfd,\mm)$ is essentially non-branching (see Definition \ref{D:essnonbranch}) and 
for every $x \in B_{4 \bar{r}_{K,N}}(\bar x)$ there exists a neighbourhood $U$ such that 
$\CD(K-\ve,N)$ is verified inside $U$. Recall that the $\CD(K,N)$ condition, introduced by Lott-Villani \cite{lottvillani:metric} and Sturm \cite{sturm:I, sturm:II}, is a synthetic notion of Ricci curvature bounded below by $K\in \R$ and dimension bounded above by $N\in [1,+\infty]$ for non smooth spaces (more precisely for metric measure spaces, i.e. metric spaces $(X,\sfd)$ endowed a non-negative measure $\mm$ playing the  role of  reference volume measure).

\begin{remark}
\begin{enumerate}
\item From Bishop volume comparison, one can easily check that  $\bar{r}_{K,N}-C_{K,N} \varepsilon \leq \bar{r}_{K-\varepsilon,N} \leq \bar{r}_{K,N}$; thus assumption (2) combined with Bishop-Gromov relative volume comparison \eqref{E:bishop} and with the volume estimate of annuli \eqref{E:mmAnnuli}, implies 
\begin{align*}
\mm(B_{\bar{r}_{K,N}}(\bar{x}))&=\mm(B_{\bar{r}_{K-\varepsilon,N}}(\bar{x}))- \Big(\mm(B_{\bar{r}_{K-\varepsilon,N}}(\bar{x}))-\mm(B_{\bar{r}_{K,N}}(\bar{x})) \Big) \nonumber\\
& \leq \left(\limsup_{r\downarrow 0} \frac{\mm(B_{r}(\bar{x}))}{\omega_{N} r^{N}} \right)  \textrm{Vol}_{K-\varepsilon,N}(\bar{r}_{K-\varepsilon,N}) +C_{K,N} \varepsilon   \nonumber\\
& \leq 1+\eta +C_{K,N} \varepsilon.
\end{align*}
 Therefore assumption (1) is  an``almost maximal volume'' condition. 
 
\item  For $N\in (1,2)$ the error in the almost euclidean isoperimetric inequality  \eqref{eq:AEII} is slightly worse; more precisely, for $N\in (1, 2)$,  \eqref{eq:AEII} is replaced by
 \begin{equation}\label{eq:AEIIN12}
 \Pe(E) \geq N \omega_{N}^{1/N} \Big(1-C_{K,N} (\delta^{\frac{2(N-1)}{N}}+\varepsilon+\eta) \Big) \; \mm(E)^{\frac{N-1}{N}}.
 \end{equation}
 
 \item In the paper we will prove directly the more general Theorem \ref{thm:MainNonSmooth}. Let us briefly comment on the fulfilment of the assumptions of  Theorem \ref{thm:MainNonSmooth} under the ones of Theorem \ref{thm:MainSmooth}.
\\First of all assumption (1) combined with Bishop-Gromov monotonicity implies that  ${\rm Vol}_{g}(B_{r}(\bar{x})) \geq (1-\eta) {\rm Vol}_{K,N} (r)$ for all $r\in (0,1)$, in particular for $r=1/4$.  Moreover,  by a standard scaling argument, it is equivalent to prove Theorem \ref{thm:MainSmooth} for the ball of radius one and for the ball of radius $4\bar{r}_{K,N}$.
\\Using geodesic normal coordinates centered in $\bar{x}$ it is readily checked that  on a smooth Riemannian manifold $(M^{N},g)$ it holds  $\lim_{r\downarrow 0} \frac{{\rm Vol}_{g}(B_{r}(\bar{x}))}{\omega_{N} r^{N}} =1$, so assumption (2) is always verified.
\\Since a Riemannian manifold is always a non-branching space, from the compatibility of the $\CD$ conditions   with the  smooth counterpart we have that assumption (3) in Theorem  \ref{thm:MainNonSmooth}
is equivalent to assumption (2) of  Theorem  \ref{thm:MainSmooth} in case the space $(X,\sfd,\mm)$ is a smooth Riemannian manifold.
\item In order to simplify the presentation we will prove the isoperimetric inequalities for the outer Minkowski content. With analogous (just slightly more technical) arguments  the same proof carries for the Perimeter (see for instance \cite{CM:Per}). Also, since the Perimeter is the relaxation of the outer Minkowski content (see for instance \cite{AGD}),  isoperimetric inequalities written in terms of the outer Minkowski content are completely equivalent to the (a priori stronger) corresponding statements written in terms of the Perimeter.
\end{enumerate}
\end{remark}

A remarkable class of metric measure spaces satisfying  Ricci curvature lower  bounds in a synthetic sense are the so called $\RCD^{*}(K,N)$-spaces, 
which include the  notable subclasses of  Alexandrov spaces with curvature bounded below (see \cite{PLSV}), and Ricci limit spaces (i.e. measured-Gromov-Hausdorff limits of Riemannian manifolds satisfying Ricci curvature lower bounds and dimension upper bounds, no matter if collapsed or not; see for instance \cite{GMS2013}).
As shown in the recent \cite{CMi}, $\RCD^{*}(K,N)$  is equivalent to $\RCD(K,N)$; 
in particular $\RCD(K,N)$-spaces are essentially non-branching (see \cite{RS2014}) and satisfy the $\CD_{loc}(K,N)$-condition and therefore the next corollary 
follows immediately from Theorem \ref{thm:MainNonSmooth}.

\begin{corollary}\label{Cor:RCD}
For every $K\in \R, N\in [2,\infty)$ there exist $\bar{\varepsilon}_{K,N}, \bar{\eta}_{K,N}, \bar{\delta}_{K,N}, C_{K,N}>0$ such that the next statement is satisfied.
Let $(X,\sfd,\mm)$ be an $\RCD(K-\varepsilon,N)$ space for some  $\varepsilon\in [0,\bar{\varepsilon}_{K,N}]$,  and let $\bar{x}\in X$. Assume that for some   $\eta \in [0,\bar{\eta}_{K,N}]$ it holds:
$$
\mm(B_{\bar{r}_{K,N}}(\bar{x})) \geq (1-\eta) {\rm{Vol}}_{K,N}({\bar{r}_{K,N}}) \quad \text{and} \quad  \limsup_{r\downarrow 0} \frac{\mm(B_{r}(\bar{x}))}{\omega_{N} r^{N}} \leq 1+\eta.
$$
Let $\delta \in (0,\bar{\delta}_{K,N}]$. Then for every finite perimeter subset $E\subset B_{\delta}(\bar{x})$ the following almost euclidean isoperimetric inequality holds:
\begin{equation}\label{eq:AEIIRCD}
 \Pe(E) \geq N \omega_{N}^{1/N} \Big(1-C_{K,N} (\delta+\varepsilon+\eta) \Big) \; \mm(E)^{\frac{N-1}{N}},
\end{equation}
where $\Pe(E)$ is the perimeter of $E$.
\end{corollary}

\subsection*{Outline of the strategy and organization of the paper} 
One of the main difficulties in proving Theorem \ref{thm:MainSmooth} relies in the fact that the classical proof of the L\'evy-Gromov isoperimetric inequality \cite[Appendix C]{Gro}  gives a lower bound on the perimeter of an \emph{isoperimetric region}; indeed in the arguments of L\'evy and Gromov it is essential that the boundary of an isoperimetric region is smooth (up to a singular part of large Hausdorff codimension) and has constant mean curvature. A first problem one faces in trying to adopt such a strategy in proving a local statement like Theorem \ref{thm:MainSmooth}  is that it is not clear whether the ball in question contains isoperimetric regions for small volumes, or if these tend to go out (or to the boundary) of the ball itself.

A way to overcome this problem is to prove a lower bound on the perimeter of \emph{every} Borel subset contained in the ball, not just for minimizers.
To this aim we employ a localization argument which roughly aims to  reduce the problem to a family of lower dimensional (in our case actually 1-dimensional) inequalities.
The localization technique has its roots in a work of Payne-Weinberger \cite{PW} and was then formalized by Gromov-Milman \cite{GrMi} and  Kannan-Lovatz-Simonovitz \cite{KaLoSi}. All these papers are set in very symmetric frameworks (like $\R^{n}$ or ${\mathbb S}^{n}$) and  the dimension reduction is performed by iterative bisections, which heavily make use of the symmetry of the space. Recently, via a completely different approach via $L^{1}$-optimal transportation, Klartag \cite{klartag} has  been able to generalize the localization method to smooth Riemannian manifolds without any symmetry assumption.  The localization via $L^{1}$-transportation has been further extended \cite{CM} by the authors of the present paper   to non smooth spaces satisfying lower Ricci curvature bounds, more precisely to essentially non branching $\CD_{loc}(K,N)$ metric measure spaces. With such a technique, combined also with the work of E. Milman \cite{Mil} on isoperimetry in weighted manifolds, in \cite{CM} the authors established  the validity of the  L\'evy-Gromov inequality in essentially non branching $\CD_{loc}(K,N)$ metric measure spaces (see also  \cite{CM2} for more applications).

In the present paper we adopt the point of view of \cite{CM}. The main technical challenge is given by the fact that while the curvature and non-branching conditions  in  \cite{CM} were assumed to hold over all the space, here they are just assumed  on a metric ball. This introduces 
some difficulties which are tackled in Section \ref{Sec:Loc}; one of the main ideas is to replace the set $\Gamma$, defined in \eqref{E:Gamma} and used in the analysis of \cite{CM},  by the set $\bar{\Gamma}_{\delta,r}$ defined  in \eqref{E:barGammadeltar}.
Such a replacement brings technical  changes in the arguments of  \cite{CM} which are presented in Section  \ref{Sec:Loc}. Thanks to this adjustment, all the arguments will be localized inside the ball $B_{4\bar{r}_{K,N}}(\bar{x})$; this is the reason why it is enough to assume \emph{local} conditions on the space. 
\\Section \ref{S:ProofMain} is finally  devoted to the proof of Theorem \ref{thm:MainNonSmooth}, which directly  implies Theorem \ref{thm:MainSmooth}  and Corollary  \ref{Cor:RCD} as already mentioned above.

\subsection*{Acknowledgment}
The authors wish to express their gratitude to  Richard Bamler  for pointing out to them  the relevance of  having an almost euclidean isoperimetric inequality under Ricci curvature lower bounds, and to Peter Topping for his interest  and keen support which added motivation to this project.

\section{Preliminaries}

In this introductory section, for simplicity of presentation, we will assume  $(X, \sfd)$ to be a complete, proper and separable metric space endowed with a  positive Radon measure $\mm$.  As explained later, since in the paper we will only work inside a metric ball $B$, the global completeness and properness can be relaxed to the assumption that the metric ball $B$ is relatively compact. The triple   $(X, \sfd,\mm)$ is said to be a metric measure space, m.m.s. for short.
\\The space of all Borel probability measures over $X$ will be denoted by $\mathcal{P}(X)$.
\\A metric space is a geodesic space if and only if for each $x,y \in X$ 
there exists $\gamma \in \Geo(X)$ so that $\gamma_{0} =x, \gamma_{1} = y$, with
$$
\Geo(X) : = \{ \gamma \in C([0,1], X):  \sfd(\gamma_{s},\gamma_{t}) = |s-t| \sfd(\gamma_{0},\gamma_{1}), \text{ for every } s,t \in [0,1] \}.
$$
Recall that, for complete geodesic spaces, local compactness is equivalent to properness (a metric space is proper if every closed ball is compact).

\medskip

We denote with $\mathcal{P}_{2}(X)$ the space of probability measures with finite second moment  endowed with the $L^{2}$-Kantorovich-Wasserstein distance  $W_{2}$ defined as follows:  for $\mu_0,\mu_1 \in \mathcal{P}_{2}(X)$ we set
\begin{equation}\label{eq:Wdef}
  W_2^2(\mu_0,\mu_1) = \inf_{ \pi} \int_{X\times X} \sfd^2(x,y) \, \pi(dxdy),
\end{equation}
where the infimum is taken over all $\pi \in \mathcal{P}(X \times X)$ with $\mu_0$ and $\mu_1$ as the first and the second marginal.
Assuming the space $(X,\sfd)$ to be geodesic, also the space $(\mathcal{P}_2(X), W_2)$ is geodesic. 

Any geodesic $(\mu_t)_{t \in [0,1]}$ in $(\mathcal{P}_2(X), W_2)$  can be lifted to a measure $\nu \in {\mathcal {P}}(\Geo(X))$, 
so that $({\rm e}_t)_\sharp \, \nu = \mu_t$ for all $t \in [0,1]$. 
Here for any $t\in [0,1]$,  ${\rm e}_{t}$ denotes the evaluation map: 
$$
  {\rm e}_{t} : \Geo(X) \to X, \qquad {\rm e}_{t}(\gamma) : = \gamma_{t}.
$$

Given $\mu_{0},\mu_{1} \in \mathcal{P}_{2}(X)$, we denote by 
$\Opt(\mu_{0},\mu_{1})$ the space of all $\nu \in \mathcal{P}(\Geo(X))$ for which $({\rm e}_0,{\rm e}_1)_\sharp\, \nu$ 
realizes the minimum in \eqref{eq:Wdef}. If $(X,\sfd)$ is geodesic, then the set  $\Opt(\mu_{0},\mu_{1})$ is non-empty for any $\mu_0,\mu_1\in \mathcal{P}_2(X)$.
It is worth also introducing the subspace of $\mathcal{P}_{2}(X)$
formed by all those measures absolutely continuous with respect with $\mm$: it is denoted by $\mathcal{P}_{2}(X,\sfd,\mm)$.

\subsection{Geometry of metric measure spaces}\label{Ss:geom}
Here we briefly recall the synthetic notions of  Ricci curvature lower bounds, for more details we refer to  \cite{lottvillani:metric,sturm:I, sturm:II, Vil}.

In order to formulate the curvature properties for $(X,\sfd,\mm)$ we introduce the following distortion coefficients: given two numbers $K,N\in \erre$ with $N\geq1$, we set for $(t,\theta) \in[0,1] \times \erre_{+}$, 
\begin{equation}\label{E:tau}
\tau_{K,N}^{(t)}(\theta):= 
\begin{cases}
\infty, & \textrm{if}\ K\theta^{2} \geq (N-1)\pi^{2}, \crcr
\displaystyle   t^{1/N} \left( \frac{\sin(t\theta\sqrt{K/(N-1)})}{\sin(\theta\sqrt{K/(N-1)})} \right)^{(N-1)/N}& \textrm{if}\ 0< K\theta^{2} <  (N-1)\pi^{2}, \crcr
t & \textrm{if}\ K \theta^{2}<0 \ \textrm{and}\ N=1, \ \textrm{or  if}\ K \theta^{2}=0,  \crcr
\displaystyle    t^{1/N} \left( \frac{\sinh(t\theta\sqrt{-K/(N-1)})}{\sinh(\theta\sqrt{-K/(N-1)})} \right)^{(N-1)/N}  & \textrm{if}\ K\theta^{2} \leq 0 \ \textrm{and}\ N>1.
\end{cases}
\end{equation}
We will also make use of the coefficients  $\sigma$ defined, for $K,N \in \R, \;N>1$, by  the identity: 
\begin{equation}\label{E:sigma}
\tau_{K,N}^{(t)}(\theta): = t^{1/N} \sigma_{K,N-1}^{(t)}(\theta)^{(N-1)/N}.
\end{equation}
As we will consider only the case of essentially non-branching spaces, we recall the following definition. 
\begin{definition}\label{D:essnonbranch}
A metric measure space $(X,\sfd, \mm)$ is \emph{essentially non-branching} if and only if for any $\mu_{0},\mu_{1} \in \mathcal{P}_{2}(X)$,
with $\mu_{0},\mu_{1}$ absolutely continuous with respect to $\mm$, any element of $\Opt(\mu_{0},\mu_{1})$ is concentrated on a set of non-branching geodesics.
\end{definition}

A set $F \subset \Geo(X)$ is a set of non-branching geodesics if and only if for any $\gamma^{1},\gamma^{2} \in F$, it holds:
$$
\exists \;  \bar t\in (0,1) \text{ such that } \ \forall t \in [0, \bar t\,] \quad  \gamma_{ t}^{1} = \gamma_{t}^{2}   
\quad 
\Longrightarrow 
\quad 
\gamma^{1}_{s} = \gamma^{2}_{s}, \quad \forall s \in [0,1].
$$
It is clear that if $(X,\sfd)$ is a  smooth Riemannian manifold  then any subset $F \subset \Geo(X)$  is a set of non branching geodesics, in particular any smooth Riemannian manifold is essentially non-branching.

\begin{definition}[$\CD$ condition]\label{D:CD}
An essentially non-branching m.m.s. $(X,\sfd,\mm)$ verifies $\mathsf{CD}(K,N)$  if and only if for each pair 
$\mu_{0}, \mu_{1} \in \mathcal{P}_{2}(X,\sfd,\mm)$ there exists $\nu \in \Opt(\mu_{0},\mu_{1})$ such that
\begin{equation}\label{E:CD}
\r_{t}^{-1/N} (\gamma_{t}) \geq  \tau_{K,N}^{(1-t)}(\sfd( \gamma_{0}, \gamma_{1}))\r_{0}^{-1/N}(\gamma_{0}) 
 + \tau_{K,N}^{(t)}(\sfd(\gamma_{0},\gamma_{1}))\r_{1}^{-1/N}(\gamma_{1}), \qquad \nu\text{-a.e.} \, \gamma \in \Geo(X),
\end{equation}
for all $t \in [0,1]$, where $({\rm e}_{t})_\sharp \, \nu = \r_{t} \mm$.
\end{definition}

For the general definition of $\CD(K,N)$ see \cite{lottvillani:metric, sturm:I, sturm:II}. It is worth recalling that if $(M,g)$ is a Riemannian manifold of dimension $n$ and 
$h \in C^{2}(M)$ with $h > 0$, then the m.m.s. $(M,\sfd_{g},h \, {\rm Vol}_{g})$ (where $\sfd_{g}$ and ${\rm Vol}_{g}$ denote the Riemannian distance and volume induced by $g$) verifies $\CD(K,N)$ with $N\geq n$ if and only if  (see Theorem 1.7 of \cite{sturm:II})
$$
\Ric_{g,h,N} \geq  K g, \qquad \Ric_{g,h,N} : =  \Ric_{g} - (N-n) \frac{\nabla_{g}^{2} h^{\frac{1}{N-n}}}{h^{\frac{1}{N-n}}}.  
$$

We also mention the more recent Riemannian curvature dimension condition $\RCD^{*}$ introduced 
in the infinite dimensional case $N = \infty$  in \cite{AGS11b} for finite reference measure $\mm$ (see  also  \cite{AGMR12} for the extension to $\sigma$-finite reference measures and for a simplification in the axiomatization).
The finite dimensional refinement lead to the  class of $\RCD^{*}(K,N)$-spaces  with $N<\infty$: a m.m.s. verifies $\RCD^{*}(K,N)$ if and only if it satisfies  $\CD^{*}(K,N)$ (see \cite{BS10}) and is 
infinitesimally Hilbertian \cite[Definition 4.19 and Proposition 4.22]{gigli:laplacian}, meaning that 
the Sobolev space $W^{1,2}(X,\mm)$ is a Hilbert space (with the Hilbert structure induced by the Cheeger energy). We refer to  \cite{gigli:laplacian, EKS,AMS, MN} for a general account 
on the synthetic formulation of the latter Riemannian-type Ricci curvature lower bounds in finite dimension.  

A remarkable property  is the equivalence of the $\RCD^{*}(K,N)$ condition and the  Bochner inequality \cite{EKS,AMS}; 
moreover the class of $\RCD^{*}(K,N)$-spaces  include the  notable subclasses of  Alexandrov spaces with curvature bounded below (see \cite{PLSV}), and Ricci limit spaces (i.e. measured-Gromov-Hausdorff limits of Riemannian manifolds satisfying Ricci curvature lower bounds and dimension upper bounds, no matter if collapsed or not; see for instance \cite{GMS2013}). It was proved in \cite{RS2014} that $\RCD^{*}(K,N)$-spaces are essentially non-branching.

Recently it has been proved the equivalence between $\RCD^{*}(K,N)$ and $\RCD(K,N)$, where the latter 
is the analogous strengthening of the Curvature Dimension condition (see \cite[Section 13.2]{CMi}): a m.m.s. verifies $\RCD(K,N)$ if and only if it satisfies $\CD(K,N)$ and is 
infinitesimally Hilbertian.  The difference between $\CD^{*}(K,N)$ and $\CD(K,N)$ being that the former (called reduced Curvature Dimension condition) is a priori weaker and more suitable to the local-to-global property (under the essential non-branching assumption), the latter has the advantage of giving sharp constants in the comparison  results  \cite{sturm:II}    directly from the definition  but a priori is less clear to satisfy the  local-to-global property.

We now state the local formulation of $\mathsf{CD}(K,N)$.

\begin{definition}[$\CD_{loc}$ condition]\label{D:loc}
An essentially non-branching m.m.s. $(X,\sfd,\mm)$ satisfies $\CD_{loc}(K,N)$ if for any point $x \in X$ there exists a neighbourhood $X(x)$ of $x$ such that for each pair 
$\mu_{0}, \mu_{1} \in \mathcal{P}_{2}(X,\sfd,\mm)$ supported in $X(x)$
there exists $\nu \in \Opt(\mu_{0},\mu_{1})$ such that \eqref{E:CD} holds true for all $t \in [0,1]$.
The support of $({\rm e}_{t})_\sharp \, \nu$ is not necessarily contained in the neighbourhood $X(x)$.
\end{definition}

Clearly $\RCD(K,N)$-spaces also satisfy the $\CD_{loc}(K,N)$-condition, for $N>1$. \\

Given a m.m.s. $(X,\sfd,\mm)$ as above and  a Borel subset $E\subset X$, let $E^{\rho}$ denote the $\rho$-tubular neighbourhood 
$$
E^{\rho}:=\{x \in X \,:\, \exists y \in E \text{ such that } \sfd(x,y) < \rho \}. 
$$
The outer Minkowski  content  $\mm^+(E)$, which should be seen as the ``boundary measure'',  is defined by
\begin{equation}\label{eq:MinkCont}
\mm^+(E):=\liminf_{\rho\downarrow 0} \frac{\mm(E^\rho)-\mm(E)}{\rho}.
\end{equation}
The isoperimetric profile function $\cI_{(X,\sfd,\mm)}:[0,1]\to \R^{+}$  of $(X,\sfd,\mm)$ is defined by
$$
\cI_{(X,\sfd,\mm)}(v):=\inf \{ \mm^{+}(E) \,:\, E \subset X \text{ Borel subset with } \mm(E)=v \}.
$$

\subsection{The model isoperimetric profile function $\cI_{K,N,D}$}\label{SS:IKND}
If $K>0$ and $N\in \N$, by the L\'evy-Gromov isoperimetric inequality  we know that, for $N$-dimensional smooth manifolds having Ricci $\geq K$, the isoperimetric profile function is bounded below by the one of the $N$-dimensional round sphere of the suitable radius. In other words  the \emph{model} isoperimetric profile function is the one of ${\mathbb S}^N$. For $N\geq 1, K\in \R$ arbitrary real numbers the situation is  more complicated, and just recently E. Milman  \cite{Mil} discovered the model isoperimetric profile function.   In this short section we recall its definition.
\\

Given $\delta>0$, set 
\[
\begin{array}{ccc}
 {\rm s}_\delta(t) := \begin{cases}
\sin(\sqrt{\delta} t)/\sqrt{\delta} & \delta > 0 \\
t & \delta = 0 \\
\sinh(\sqrt{-\delta} t)/\sqrt{-\delta} & \delta < 0
\end{cases}

& , &

 {\rm c}_\delta(t) := \begin{cases}
\cos(\sqrt{\delta} t) & \delta > 0 \\
1 & \delta = 0 \\
\cosh(\sqrt{-\delta} t) & \delta < 0
\end{cases}
\end{array} ~.
\]
Given a continuous function $f :\R \to  \R$ with $f(0) \geq 0$, we denote with $f_+ : \R \to \R^+ $ the function coinciding with $f$ between its first non-positive and first positive roots, and vanishing everywhere else, i.e. $f_+ := f \chi_{[\xi_{-},\xi_{+}]}$ with $\xi_{-} = \sup\{\xi \leq 0; f(\xi) = 0\}$ and $\xi_{+} = \inf\{\xi > 0; f(\xi) = 0\}$.

Given $H,K \in \R$ and $N \in [1,\infty)$, set $\delta := K / (N-1)$  and define the following (Jacobian) function of $t \in \R$:
\[
{\mathcal J}_{H,K,N}(t) :=
\begin{cases}
\chi_{\{t=0\}}  & N = 1 , K > 0 \\
\chi_{\{H t \geq 0}\} & N = 1 , K \leq 0 \\
\left({\rm c}_\delta(t) + \frac{H}{N-1} {\rm s}_\delta(t)\right)_+^{N-1} & N \in (1,\infty) \\
\end{cases} ~.
\]
As last piece of notation, given a non-negative integrable function $f$ on a closed interval $L \subset \R$, we denote with $\mu_{f,L}$  
the probability measure supported in $L$ with density (with respect to the Lebesgue measure) proportional to $f$ there. In order to simplify a bit the notation we will write
$\cI_{(L,f)}$ in place of $\cI_{(L,\, |\cdot|, \mu_{f,L})}$.
\\The model isoperimetric profile for spaces having Ricci $\geq K$, for some $K\in \R$, dimension bounded above by $N\geq 1$ and diameter at most $D\in (0,\infty]$ is then defined by
\begin{equation}\label{eq:defIKND}
\cI_{K,N,D}(v):=\inf_{H\in \R,a\in [0,D]} \cI_{\left([-a,D-a], {\mathcal J}_{H,K,N}\right)} (v), \quad \forall v \in [0,1]. 
\end{equation}
The formula above has the advantage of considering all the possible cases in just one equation, 
but  probably it is  also useful to  isolate the different cases in a more explicit way. Indeed one can check \cite[Section 4] {Mil} that:
\begin{itemize}

\item \textbf{Case 1}: $K>0$ and $D<\sqrt{\frac{N-1}{K}} \pi$,
$$
\cI_{K,N,D}(v) =   \inf_{\xi \in \big[0, \sqrt{\frac{N-1}{K}} \pi -D\big]} \cI_{\big( [\xi,\xi +D],  \sin( \sqrt{\frac{K}{N-1}} t)^{N-1} \big)}(v), \quad \forall v \in [0,1] ~. 
$$

\item \textbf{Case 2}:   $K > 0$ and $D \geq \sqrt{\frac{N-1}{K}} \pi$, 
$$
\cI_{K,N,D}(v) = \cI_{ \big( [0, \sqrt{\frac{N-1}{K}} \pi],  \sin( \sqrt{\frac{K}{N-1}} t)^{N-1}   \big)}(v), \quad \forall v \in [0,1] ~. 
$$

\item  \textbf{Case 3}: $K=0$ and $D<\infty$,
\begin{eqnarray*}
\cI_{K,N,D}(v)&=&  \min \left \{ \begin{array}{l}  \inf_{\xi \geq 0} \cI_{([\xi,\xi+D],  t^{N-1})}(v) ~,\\
\phantom{\inf_{\xi \in \R}} \cI_{([0,D],1)}(v)
\end{array}
\right \} \\
&=& \frac{N}{D} \inf_{\xi \geq 0}  \frac{\left(\min(v,1-v) (\xi+1)^{N} + \max(v,1-v) \xi^{N}\right)^{\frac{N-1}{N}}}{(\xi+1)^{N} - \xi^{N}},  \quad \forall v \in [0,1] ~. 
\end{eqnarray*}

\item \textbf{Case 4}:  $K < 0$, $D<\infty$:
$$
\cI_{K,N,D}(v)=  \min \left \{ \begin{array}{l}
\inf_{\xi \geq 0} \cI_{\big([\xi,\xi+D], \; \sinh(\sqrt{\frac{-K}{N-1}} t)^{N-1}\big)}(v) ~ , \\
\phantom{\inf_{\xi \in \R}} \cI_{\big([0,D],   \exp(\sqrt{-K (N-1)} t) \big)} (v) ~, \\
\inf_{\xi \in \R} \cI_{\big([\xi,\xi+D], \; \cosh(\sqrt{\frac{-K}{N-1}} t)^{N-1}\big)} (v)
\end{array}
\right \} \quad \forall v \in [0,1]  ~.
$$
\item In all the remaining cases,  that is for $K\leq 0, D=\infty$,  the model profile trivializes: $\cI_{K,N.D}(v)=0$ for every $v \in [0,1].$
\end{itemize}
\medskip
Note that when $N$ is an integer, $$\cI_{\big( [0,  \sqrt{\frac{N-1}{K}} \pi ], ( \sin(\sqrt{\frac{K}{N-1} } t)^{N-1}\big)} = \cI_{({\mathbb S}^{N}, g^K_{can}, \mu^K_{can})}$$
 by the isoperimetric inequality on the sphere, and so Case 2 with $N$ integer corresponds to L\'evy-Gromov isoperimetric inequality. 
 \\From the definition of $\cI_{K,N,D}$  and of $\bar{r}_{K,N}$ one can also check that
\begin{equation}\label{eq:IKNvN-1Nprel}
\cI_{K,N,\bar{r}_{K,N}} (v)  = N \omega_{N}^{1/N}  \; v^{\frac{N-1}{N}}+O_{K,N}(v^{\frac{3(N-1)}{N}}), \quad \forall v \in [0, 1], 
\end{equation}
where $(0,1]\ni t\mapsto O_{K,N}(t) \in \R$ is a smooth function depending smoothly on $K,N$ and satisfying 
\begin{equation}\label{eq:limsupoKNprel}
\limsup_{t\downarrow 0} \frac{|O_{K,N}(t)|}{t}< \infty.
\end{equation}
To this aim observe  that, for small volumes $v$, the minimum in Case 3  is achieved by    $\inf_{\xi \geq 0} \cI_{([\xi,\xi+\bar{r}_{0,N}],  t^{N-1})}(v)$ which for small $v$ coincides with $\cI_{([0, \bar{r}_{0,N}],  t^{N-1})}(v)$,
and in Case 4 the minimum is achieved by  $\inf_{\xi \geq 0} \cI_{\big([\xi,\xi+\bar{r}_{K,N}], \; \sinh(\sqrt{\frac{-K}{N-1}} t)^{N-1}\big)}(v)$ which for small $v$ coincides with $\cI_{\big([0, \bar{r}_{K,N}], \; \sinh(\sqrt{\frac{-K}{N-1}} t)^{N-1}\big)}(v)$. 
The two claims \eqref{eq:IKNvN-1Nprel} and \eqref{eq:limsupoKNprel} follow then  from  such  explicit expressions (and from Case 2, if $K>0$) of $\cI_{K,N,\bar{r}_{K,N}} (v)$ for small $v$.  
\\From such explicit expressions it is also clear that the map  $\R\times \R_{>1} \times \R_{>0}\ni(K,N,D) \mapsto \cI_{K,N,D}(v)$ for fixed $v>0$ sufficiently small  is smooth (actually $C^{1}$ will be enough for our purposes). 

Note that when  $N\in \N, N\geq 2$, for small $v>0$ it holds $\cI_{K,N,\bar{r}_{K,N}}(v)=|\partial B^{{\mathbb M}^{N}_{K/(N-1)}}(v)|$, where $|\partial B^{{\mathbb M}^{N}_{K/(N-1)}}(v)|$  denotes the $N-1$-dimensional boundary measure of the metric ball of volume $v$ inside the model space ${\mathbb M}^{N}_{K/(N-1)}$; in this case the two claims \eqref{eq:IKNvN-1Nprel} and \eqref{eq:limsupoKNprel}  follow then by standard computations in geodesic normal coordinates.

\section{$L^{1}$-localization method}\label{Sec:Loc}
For the rest of the paper  $(X,\sfd,\mm)$ will  be  an essentially non-branching geodesic  metric space  endowed with a non-negative Borel measure. For a fixed $\bar{x}\in X$, we will  assume  that $B_{4\bar{r}_{K,N}}(\bar{x})$ is relatively compact and that $B_{4\bar{r}_{K,N}}(\bar{x})\subset \supp(\mm), \mm(B_{4\bar{r}_{K,N}}(\bar{x}))< \infty$. 
   Thanks to Lemma \ref{lem:2delta} below, all our arguments will be related to the ball $B_{4\bar{r}_{K,N}}(\bar{x})$ (for sake of clarity let us mention that for the first part of the arguments after Lemma \ref{lem:2delta}, considering the ball  $B_{2\bar{r}_{K,N}}(\bar{x})$ would suffice; the larger ball $B_{4\bar{r}_{K,N}}(\bar{x})$ will come up naturally later in \eqref{eq:Ball4r}). In particular we will not need assumptions on the geometry of $(X,\sfd,\mm)$ outside of $B_{4\bar{r}_{K,N}}(\bar{x})$, such as (global) properness and completeness. Also the assumption that $(X,\sfd)$ is a geodesic space can be further relaxed by asking that any two points inside $B_{\bar{r}_{K,N}}(\bar{x})$ are joined by a geodesic in $(X,\sfd)$, and the assumption that $(X,\sfd,\mm)$ is essentially non-branching can be further relaxed to ask that  for any $\mu_{0},\mu_{1} \in \mathcal{P}_{2}(X)$,
with $\mu_{0},\mu_{1}$ absolutely continuous with respect to $\mm$ and $\supp(\mu_{0}), \supp(\mu_{1}) \subset B_{\bar{r}_{K,N}}(\bar{x})$, any element of $\Opt(\mu_{0},\mu_{1})$ is concentrated on a set of non-branching geodesics.
\\

Recall that we denote with $\bar{r}_{K,N}>0$ the unique radius such that ${\rm Vol}_{K,N}(\bar{r}_{K,N})=1$. For ease of notation we write $\bar{r}$ in place of $\bar{r}_{K, N}$ and 
we assume that 
\begin{equation}\label{eq:AssB}
\mm(B_{\bar{r}}(\bar{x})) \geq 1-\eta \quad \text{and} \quad B_{4\bar{r}}(\bar{x}) \text{ is  $\CD_{loc}(K-\varepsilon,N)$},
\end{equation}
for some $\varepsilon \in [0, \bar{\varepsilon}_{K,N}]$ with  $\bar{\varepsilon}_{K,N}\leq \frac {\bar{r}}{10}$ to be determined later just depending on $K,N$. 

We also fix the rescaled measure
\begin{equation}\label{eq:defbarmm}
\bar{\mm}:=\frac{1}{\mm(B_{\bar{r} + 2 \delta}(\bar{x}))}  \; \mm \llcorner  B_{\bar{r} + 2 \delta}(\bar{x}).
\end{equation}
where $\delta\in [0, \bar{\delta}_{K,N}]$, with $\bar{\delta}_{K,N}\leq \frac{\bar{r}}{10}$ to be fixed later depending just on $K,N$.

In order to localize the almost-euclidean isoperimetric inequality, 
we will consider an $L^{1}$-optimal transportation problem between probability measures $\mu_{0},\mu_{1}$ absolutely continuous with respect to $\bar \mm$;
in particular we will consider 
\begin{equation}\label{E:measures}
\supp(\mu_{0}) \subset \overline{B_{\delta}(\bar x)}\quad \text{and}\quad \supp(\mu_{1}) \subset \overline{B_{\bar r}(\bar x)}.
\end{equation}
The previous choice of $\bar \mm$ is indeed motivated by the following simple fact.

\begin{lemma}\label{lem:2delta}
Let $(X,\sfd)$ be a metric space and fix  $\delta \leq \bar{r}/2$, $\bar{x}\in X$. For any $x \in B_{\delta}(\bar{x}), y \in B_{\bar{r}}(\bar{x})$, any geodesic  from $x$ to $y$ has length at most $\bar{r}+\delta$ and  is contained in $B_{\bar{r}+2 \delta}(\bar{x})$.
\end{lemma} 

\begin{proof}
Let $\gamma_{xy}$ be any  geodesic with $\gamma_{xy}(0)=x \in B_{\delta}(\bar{x})$ and  $\gamma_{xy}(1)=y \in B_{\bar{r}}(\bar{x})$. Then by  the triangle inequality  we have
\begin{align*}
\textrm{Length}(\gamma_{xy}) = \sfd(x,y) \leq \sfd(x,\bar{x})+\sfd(\bar{x},y) \leq \delta+ \bar{r}.
\end{align*}
If $\bar x$ is a point of $\gamma_{xy}$, then the second part of the claim follows, otherwise for every $t \in [0,1]$:
\begin{align*}
\sfd(\bar{x}, \gamma_{xy}(t)) < \sfd(\bar{x}, x)+ \sfd(x, \gamma_{xy}(t)) \leq \delta+ \textrm{Length}(\gamma_{xy})  \leq 2 \delta+ \bar{r}.
\end{align*}
\end{proof}

Since $\mm(B_{2\bar r}(\bar x)) < \infty$, both $\mu_{0}$ and $\mu_{1}$ have finite first moment; then from Kantorovich duality (see \cite[Theorem 5.10]{Vil}) 
there exists a $1$-Lipschitz function $\f: X \to R$ such that $\bar \pi \in \Pi(\mu_{0},\mu_{1})$ is $L^{1}$-optimal, i.e. for any $\pi \in \Pi(\mu_{0},\mu_{1})$
$$
\int_{X} \sfd(x,y) \bar \pi (dxdy) \leq \int_{X} \sfd(x,y) \pi (dxdy),  
$$
if and only if $\bar \pi (\Gamma) = 1$, where 
\begin{equation}\label{E:Gamma}
\Gamma : = \{ (x,y) \in X\times X \colon \f(x) - \f(y) = \sfd(x,y)  \}.
\end{equation}
From Kantorovich duality (see \cite[Theorem 5.10]{Vil}) we also deduce the existence of an optimal $\bar \pi \in \Pi(\mu_{0},\mu_{1})$.
The set $\Gamma$ permits to perform a well-known dimensional reduction of $\bar \pi$ obtaining a family of one-dimensional transport plans.
Such a reduction has a very long history and we refer to \cite{CM,klartag} for more details.


In order to obtain an almost Euclidean isoperimetric inequality, we specialize the general $L^{1}$-optimal transportation problem as follows.
Given a generic Borel subsets $E \subset B_{\delta}(\bar{x}) \subset B_{\bar r}(\bar x)$  with $\bar{\mm}(E)>0$, we look for  a ``localization'' of
the following function with zero mean
\begin{equation}\label{eq:deffA}
f_{E}(x):=\chi_{E}(x)-\frac{\bar{\mm}(E)}{\bar{\mm}(B_{\bar{r}} (\bar{x}))}  \chi_{B_{\bar{r}} (\bar{x})} (x);
\end{equation}
note indeed that trivially $\int f_{E} \bar \mm = 0$. Denoting $f^{+}_{E}:=\max\{f_{E}, 0\},  f^{-}_{E}:=\max\{-f_{E}, 0\}$ the positive and negative parts of $f_{E}$, set
$$
\int_{X} f_{E}^{+} \, \bar{\mm} =  \int_{X} f_{E}^{-} \, \bar{\mm}=: c_{E}>0.
$$
Then to obtain a localization of the aforementioned function, one studies the $L^{1}$ optimal transportation problem between the following marginal measures
\begin{equation}\label{eq:mu0mu1}
\mu_{0}:=\frac{1}{c_{E}} f_{E}^{+} \bar{\mm}  \in \mathcal{P}(X),  \; \mu_{1}:=\frac{1}{c_{E}} f_{E}^{-} \bar{\mm} \in \mathcal{P}(X), \quad \text{with } \mu_{0}(E)=1=\mu_{1}(B_{\bar{r}} (\bar{x})\setminus E).
\end{equation}

From now on we will concentrate our analysis on this particular case and we consider fixed $E \subset B_{\delta}(\bar x)$ and an associated Kantorovich potential $\f:X \to \R$ 
for the $L^{1}$ optimal transportation problem from $\mu_{0}$ to $\mu_{1}$.
To localize the analysis, is it convenient to consider the following subset of the $\Gamma$ (recall \eqref{E:Gamma}):
$$
\Gamma_{\delta,r} : = \Gamma \cap \overline{B_{\delta}(\bar x)} \times \overline{B_{\bar r}(\bar x)}. 
$$
To keep the ray structure, we fill in possible gaps of $\Gamma_{\delta,r}$ by considering the set $\bar \Gamma_{\delta,r}$ defined as follows: 
\begin{equation}\label{E:barGammadeltar}
\bar \Gamma_{\delta,r} : = \{ (\gamma_{s},\gamma_{t}) \colon \gamma \in \Geo, 0\leq s \leq t \leq 1, (\gamma_{0},\gamma_{1}) \subset \Gamma_{\delta,r} 
\}.
\end{equation}
It is readily verified that $\bar \Gamma_{\delta,r} \subset \Gamma$ and it is a closed set. Then to keep notation simple we will denote $\bar \Gamma_{\delta,r}$ simply 
with $\Gamma$.

\begin{remark}\label{R:partialorder}
Note that $\Gamma$ induces a partial order relation: if $(x,y), (y,z) \in \Gamma$ then $ (x,z)\in \Gamma$; indeed
from the definition of $\Gamma$, there exists $\gamma^{1}, \gamma^{2} \in \Geo(X)$  and $s_{1},s_{2},t_{1},t_{2} \in [0,1]$ with $s_{i} \leq t_{i}$ for $i = 1,2$ such that 
$$
(\gamma_{0}^{1},\gamma_{1}^{1}), (\gamma_{0}^{2},\gamma_{1}^{2}) \in \Gamma_{\delta,r},\quad \gamma_{s_{1}}^{1} = x,\gamma_{t_{1}}^{1} = y,  \gamma_{s_{2}}^{2} = y, \gamma_{t_{2}}^{2}=z.
$$
Since
$$
\sfd(x,y)+\sfd(y,z) = \f(x) - \f(y)+ \f(y)-\f(z) = \f(x) - \f(z) = \sfd(x,z),
$$
it follows that gluing (and reparametrizing) $\gamma^{1}\llcorner_{[0,t_{1}]}$ with $\gamma^{2}\llcorner_{[s_{2},1]}$ produces a constant speed geodesic $\eta$ such that 
$(\eta_{0},\eta_{1}) \in \Gamma_{\delta,r}$ and with $\eta_{s} = x$ and $\eta_{t} = z$ for some $0 \leq s \leq t \leq 1$. Hence $(x,z) \in \Gamma$ and the claim follows.
\end{remark}

Next we consider the classical \emph{transport relation} 
$$R : = \Gamma \cup \Gamma^{-1}, \quad   \text{where } \Gamma^{-1} = \{(x,y) \in X\times X \colon (y,x) \in \Gamma\}$$
 and the associated \emph{transport set}
$$
\T: =  P_{1} (R \setminus \{ x = y\}).
$$
Since $B_{4\bar{r}}(\bar{x})$ is relatively compact, $\Gamma$ is $\sigma$-compact (countable union of compact sets) and therefore $\T$ is $\sigma$-compact as well.

\subsection{Branching structure} 

We now remove the set of branching points so to obtain a partition of $\T$ into a family of one-dimensional sets
isometric to real intervals. 


The \emph{set of branching points} is formed by the union of the following sets
\begin{align*}
A_{+} : =&~ \{ x\in \T \colon \exists z,w \in \T, (x,z), (x,w) \in \Gamma, (z,w) \notin R \}, \\
A_{-}: = &~ \{ x\in \T \colon \exists z,w \in \T, (x,z), (x,w) \in \Gamma^{-1}, (z,w) \notin R\}.
\end{align*}
It is not difficult to check that $A_{\pm}$ are projections of $\sigma$-compact sets. The set of branching points will be denoted with $A$.

\begin{remark}\label{R:branch}
We observe the following simple fact about branching points whose proof follows straightforwardly from Remark \ref{R:partialorder}: 
if $x \notin A_{+}$ and $(x,y) \in \Gamma$, then $y \notin A_{+}$. 
The symmetric statement is also valid: if $x \notin A_{-}$ and $(y,x) \in \Gamma$, then $y \notin A_{-}$.
\end{remark}

We introduce the following notation: $\Gamma(x) : = \{ y \in X \colon (x,y) \in \Gamma \}$. More in general, if $F \subset X \times X$, $F(x) : = P_{2}(F \cap \{x\} \times X)$, where 
$P_{2}$ is the projection on the second entry.

\begin{lemma}\label{L:branch}
For each $x \in \T \setminus A_{+}$, the map $\f : \Gamma(x) \to \R$ is an isometry.
\end{lemma}

\begin{proof}
Consider $y,z \in  \Gamma(x) $; in particular 
$$
\f(x) - \f(y) = \sfd(x,y), \qquad \f(x) - \f(z) = \sfd(x,z),
$$
and with no loss in generality we assume 
$\sfd(x,y) \leq  \sfd(x,z)$, so that $\f(x) \geq  \f(y) \geq \f(z)$.
Since the space is geodesic, there exists $\gamma \in \Geo(X)$ with $\gamma_{0} = x$ and $\gamma_{1} = z$. Then one easily verifies that $(\gamma_{s},\gamma_{t}) \in \Gamma$ for any 
$0 \leq s \leq t \leq 1$; in particular, by continuity there exists $t \in [0,1]$ such that $\f(\gamma_{t}) = \f(y)$ and $\gamma_{t} \in \Gamma(x)$.
Since $x \notin A_{+}$, necessarily $\gamma_{t} = y$. It follows that 
$$
\f(y) - \f(z) = \sfd(y,z),
$$
and the claim follows.
\end{proof}

Replacing the set of forward branching points with the set of backward branching  points, one can extend Lemma \ref{L:branch} and obtain the next 

\begin{corollary}\label{C:raybranch}
For each $x \in \T \setminus A$, the map $\f : R(x)  \to \R$ is an isometry.
\end{corollary}

\begin{proof}
It has already been shown that if $x \in \T \setminus A$, then $\f : \Gamma(x) \to \R$ and $\f : \Gamma^{-1}(x)  \to \R$ are both isometries. 
To conclude note that for $y,z \in R(x) $ say with $y \in \Gamma^{-1}(x)$ and $z \in \Gamma(x)$
$$
\sfd(y,z) \geq \f(y) - \f(z) = \f(y) - \f(x) + \f(x) - \f(z) = \sfd(y,x) + \sfd(x,z) \geq \sfd(y,z),
$$
yielding the claim.
\end{proof}

We can then define the \emph{transport set without branching points} $\T^{b} := \T \setminus A$; clearly, the set $\T^{b}$ is Borel.

\begin{corollary}\label{C:equivalence}
 The set $R$ is an equivalence relation over $\T^{b}$.
\end{corollary}
\begin{proof}
The only property that requires a proof is the transitivity. So consider $x,y,z \in \T^{b}$ such that $(x,y), (y,z) \in R$. 
Since $y \in \T^{b}$ it follows that from Corollary \ref{C:raybranch} 
that $\f $ restricted to $R(y)$ is an isometry; since $x,z \in R(y)$, necessarily $|\f(x) -\f(z)| = \sfd(x,z)$, and therefore $(x,z) \in R$ and the claim follows.
\end{proof}

\begin{lemma}\label{L:convexity}
For any $x \in \T^{b}$, any $z,w \in \T^{b} \cap R(x)$ and $\gamma \in \Geo(X)$ with $\gamma_{0} =z$ and $\gamma_{1} = w$;  then for any $t \in (0,1)$ 
it holds $\gamma_{t} \in \T^{b} \cap R(x)$.  
In particular, for any $x \in \T^{b}$ the equivalence class $\T^{b} \cap R(x)$ is isometric to an interval.
\end{lemma}

\begin{proof}
First of all from Corollary \ref{C:equivalence} we know that $(z,w)\in R$. Without any loss in generality, we can assume $\f(z) \geq \f(w)$.   
It follows that   $(z,w) \in \Gamma$ and hence  $(z,\gamma_{t}) \in \Gamma$.
Since $z \notin A_{+}$, we deduce from Remark \ref{R:branch} that $\gamma_{t} \notin A_{+}$; for the same reason, from $w \notin A_{-}$ it follows that $\gamma_{t} \notin A_{-}$, 
and therefore $\gamma_{t} \in \T^{b}$. To conclude, from Corollary \ref{C:equivalence} we get that $\gamma_{t} \in R(x)$ and the claim is proved.
\end{proof}

\bigskip

It will be convenient to consider also the set $R^{b} : = R \cap \T^{b}\times \T^{b}$, introduced in \cite{CMi}; in particular, $R^{b}$ is Borel.
From Corollary \ref{C:equivalence} and Lemma \ref{L:convexity} we deduce that the equivalence classes of $R^{b}$ 
induce a partition of $\T^{b}$ into a family of subsets, each of them
isometric to an interval. Each element of the partition will be called \emph{transport ray} 
(note the difference with \cite{CMi} where sets called transport rays also satisfied a maximality property) and we will denote such family of sets with $\{ X_{q}\}_{q \in Q}$
where $Q$ is a set of indices; note that by construction $X_{q} \subset B_{\bar r+2\delta} (\bar x)$ and, by definition,
$R^{b}(x) = X_{q}$ for every $x \in X_{q}$.

Once a partition of $\T^{b}$ is given, it is possible to associate a corresponding decomposition of $\bar \mm\llcorner_{\T^{b}}$ into conditional measures, called \emph{disintegration}.  
To obtain such a disintegration of  $\bar \mm\llcorner_{\T^{b}}$ one needs first to induce a reasonable measurable structure over the set of indices $Q$.
This can be obtained constructing a \emph{section} for $R^{b}$: 
a map $f : \T^{b} \to \T^{b}$ is a section for $R^{b}$ if and only if $(x,f(x)) \in R^{b}$ and $f(x) = f(y)$ whenever $(x,y) \in R^{b}$. 
Once a section $f$ is given, one can consider as quotient set $Q: = \{ x \in \T^{b} \colon x = f(x) \}$ and then $f : \T^{b} \to Q$  will be called quotient map.
The following Lemma (already presented in \cite{CM}) permits to construct an explicit section. 
Since the equivalence relation in the present setting is slightly different from \cite{CM}, as here we consider a localized problem, 
for completeness of exposition we include below the detailed construction.
The $\sigma$-algebra generated by analytic sets is denoted by $\mathcal{A}$.

\begin{lemma}\label{lem:Qlevelset}
There exists an $\mathcal{A}$-measurable section $\QQ: \T^{b} \to \T^{b}$ such that the quotient set $Q = \{ x = \QQ(x) \}$ is $\mathcal{A}$-measurable and 
can be written locally  as a level set of $\f$ in the following sense: 
$$
Q = \bigcup_{i\in \N} Q_{i}, \qquad Q_{i} \subset \f^{-1}(\alpha_{i}) \cap \T^{b}, 
$$
for some suitable $\alpha_{i} \in \R$, with $Q_{i}$ $\mathcal{A}$-measurable and $Q_{i} \cap Q_{j} = \emptyset$, for $i\neq j$.
\end{lemma}

\begin{proof}

{\bf Step 1.} Construction of the quotient map.\\
For each $n \in \N$, consider the set $\T^{b}_{n}$ of those points $x$ having ray $R^{b}(x)$ longer than $1/n$, i.e.
$$
\T^{b}_{n} : = P_{1} \{ (x,y) \in \T^{b} \times \T^{b} \cap R^{b} \colon \sfd(x,y) \geq 1/n \}.
$$
It is easily seen that  $\T^{b}=\bigcup_{n \in \N} \T_n^{b}$ and that  $\T_{n}^{b}$ is analytic; moreover if $x \in \T_{n}^{b}, y \in \T^{b}$ and $(x,y) \in R^{b}$ then also $y \in \T_{n}^{b}$. 
In particular, $\T_{n}^{b}$ is the union of all those rays of $\T^{b}$ with length at least $1/n$. 

Now we consider the following saturated subsets of $\T_{n}^{b}$:  for $\alpha \in \R$
\begin{equation}\label{eq:defTnalpha}
 \T_{n,\alpha}^{b}:=  P_{1}  \Big( R^{b} \cap \Big \{ (x,y) \in \T_{n}^{b} \times \T_{n}^{b} \colon  \f(y) = \alpha - \frac{1}{3n}\Big \}  \Big)  \cap P_{1} \Big( R^{b} \cap \Big \{ (x,y) \in \T_{n}^{b} \times \T_{n}^{b} 
 \colon  \f(y) = \alpha+  \frac{1}{3n} \Big\}  \Big),
\end{equation}
and we claim  that for each $n\in \N$ there exists a sequence $\{\alpha_{i,n}\}_{i\in \N} \subset \R$ such that 
\begin{equation}  \label{eq:Tnalpha}
\T_{n}^{b} =  \bigcup_{i \in \N} \T_{n,\alpha_{i,n}}^{b}.
\end{equation}
First note that $(\supset)$ holds trivially. For the converse inclusion $(\subset)$   observe that
for each $\alpha \in \Q$, the set $ \T_{n,\alpha}^{b}$ coincides with the family of those $x \in \T_{n}^{b}$ 
such that there exists $y^{+},y^{-} \in R^{b}(x)$ such that
\begin{equation}\label{eq:ypm}
\f(y^{+}) = \alpha - \frac{1}{3n}, \qquad \f(y^{-}) = \alpha + \frac{1}{3n}. 
\end{equation}
Since $x \in \T_{n}^{b}$, $R^{b}(x)$ is longer than $1/n$ and therefore there exist $z,y^{+},y^{-} \in R^{b}(x) \cap \T_{n}^{b}$ such that
$$
\f(z) -\f(\bar y^{+}) = \frac{1}{2n}, \qquad  \f(\bar y^{-}) -\f(z)= \frac{1}{2n}. 
$$
From Lemma \ref{L:convexity}, it is then not hard to find $y^{-}, y^{+} \in R^{b}(x)$ and $\alpha_{i,n} \in \Q$ satisfing \eqref{eq:ypm} and therefore $x\in  \T_{n,\alpha_{i,n}}^{b}$, 
proving \eqref{eq:Tnalpha}.
By the above construction, one can check that for each $l \in \Q$, the  level set $\f^{-1}(l)$ is a quotient set for  
$\T^{b}_{n,l}$, i.e.
$\T^{b}_{n,l}$ is formed by disjoint geodesics  each one  intersecting 
$\f^{-1}(l)$ in exactly one point. 

After taking differences between all the $\T_{n,l}$, 
we end up with 
$$
\T^{b} = \cup_{n} \mathcal{K}_{n}, 
$$
with $\{ \mathcal{K}_{n} \}_{n \in \N}$ disjoint family of 
$\mathcal{A}$-measurable sets, 
saturated with respect to $R^{b}$ and, for each $n$, 
$\mathcal{K}_{n} \subset \T^{b}_{n,l_{n}}$ for some $l_{n} \in \Q$.
We can therefore define a quotient map $\QQ:\T^{b} \to Q$ 
by defining its graph: 
$$
\gr(\QQ) = \bigcup_{n \in \N} \mathcal{K}_{n} \times \left( \f^{-1}(l_{n}) \cap 
\T^{b}_{n,l_{n}} \right) \cap R^{b}.
$$

\smallskip
{\bf Step 2.} Measurability. \\
To prove the claim it will by sufficient to show that for any $U \subset X$ 
open set, $\QQ^{-1}(U) \in \mathcal{A}$; for ease of notation pose 
$\gr(\QQ) = \Lambda$:
\begin{align*}
\QQ^{-1}(U) = &~ P_{1} (\Lambda \cap X \times U)  \\
= &~ P_{1} \Big(\cup_{n} \mathcal{K}_{n} \times (\f^{-1}(l_{n}) \cap \T^{b}_{n,l_{n}}) \cap X \times U \cap R^{b}\Big) \\ 
= &~ P_{1} \Big(\cup_{n} \mathcal{K}_{n} \times (\f^{-1}(l_{n}) \cap \T^{b}_{n,l_{n}} \cap  U)  \cap R^{b}\Big) \\
= &~ \cup_{n} P_{1} \Big( \mathcal{K}_{n} \times (\f^{-1}(l_{n}) \cap \T^{b}_{n,l_{n}} \cap  U)\cap R^{b} \Big) \\
= &~ \cup_{n} \mathcal{K}_{n}\cap \Big( R_{u}^{b}(\f^{-1}(l_{n}) \cap \T^{b}_{n,l_{n}} \cap  U)\Big),
\end{align*}
where the last identity follows from the saturation property of $\mathcal{K}_{n}$ with respect to $R_{u}^{b}$: 
$$
\{ x \in \mathcal{K}_{n} \colon \exists y \in \f^{-1}(l_{n}) \cap \T^{b}_{n,l_{n}} \cap U,  (x,y) \in R_{u}^{b}\} =  
\mathcal{K}_{n}\cap \Big( R_{u}^{b}(\f^{-1}(l_{n}) \cap \T^{b}_{n,l_{n}} \cap  U) \Big).
$$
Moreover, $\f^{-1}(l_{n})$ is closed, $\T_{n,l_{n}}$ analytic and $R_{u}^{b}$ Borel, hence
$$
R_{u}^{b}(\f^{-1}(l_{n}) \cap \T^{b}_{n,l_{n}} \cap  U) = 
P_{2} \Big( R_{u}^{b} \cap [\f^{-1}(l_{n}) \cap \T^{b}_{n,l_{n}} \cap  U] 
\times X \Big)
$$
is an analytic set, showing that $\QQ^{-1}(U)$ belongs to $\mathcal{A}$. 
The claim follows.
\end{proof}

We next apply the Disintegration Theorem in order to decompose the $\bar \mm\llcorner_{\T^{b}}$ according to the quotient map $\QQ$ constructed in Lemma  \ref{lem:Qlevelset}. We  will follow \cite[Appendix A]{biacar:cmono} where a  self-contained approach (and a proof) of the Disintegration Theorem in countably generated measure spaces can be found. 
An even more general version of the Disintegration Theorem can be found in \cite[Section 452]{Fre:measuretheory4}; for preliminaries see also  \cite[Section 6.3]{CMi}.
\\The Disintegration Theorem combined with Lemma \ref{lem:Qlevelset} gives (see also \cite[Theorem 6.18]{CMi} and subsequent discussions for more details) the existence of a measure-valued map 
$$Q \ni q \to \bar \mm_{q} \in \mathcal{P}(B_{\bar r + 2\delta}(\bar x))$$
 such that:
 \begin{itemize}
 \item  for each Borel set 
$C \subset B_{\bar r + 2\delta}(\bar x)$,
the map $q \mapsto \bar\mm_{q}(C)$ is $\qq$-measurable, where $\qq = \QQ_{\#} \bar \mm\llcorner_{\T^{b}}$ is a Borel measure;
\item  the following formula holds true:
\begin{equation}\label{E:disintegration}
\bar \mm\llcorner_{ \T^{b}} = \int_{Q} \bar \mm_{q} \, \qq(dq);
\end{equation}
\item for $\qq$-a.e. $q \in Q$, $\bar \mm_{q}(X_{q}) = 1$.
\end{itemize}
%


\subsection{Negligibility of branching points and $\CD(K-\ve,N)$-disintegration}

The next step in order to obtain the localization is to prove that the set of branching points is $\bar \mm$-negligible. 
This fact can be proved, by a well-known argument already presented in \cite{cava:MongeRCD} and used many times \cite{CMi,CM,CM2},
as a consequence of the existence and uniqueness of the optimal transport map for the $L^{2}$ optimal transportation problem between probability 
measures $\mu_{0},\mu_{1}$  with $\mu_{0}$ absolutely continuous with respect to the ambient measure. 

Existence and uniqueness 
have been proved in \cite{CM3} to hold in the framework of essentially non-branching, length spaces verifying $\CD_{loc}(K,N)$. 
The same argument can be used in our framework; notice indeed that from Lemma \ref{lem:2delta} 
$$
\T\subset B_{\bar r + 2\delta}(\bar x); 
$$
therefore, any $\nu \in \mathcal{P}(\Geo(X))$ such that $[0,1] \ni t \to ({\rm e}_{t})_{\sharp} \nu$ is a $W_{2}$-geodesic,  with $({\rm e}_{i})_{\sharp}\nu (\T) =1$ for $i=0,1$ 
have support contained in the larger ball $B_{4\bar r}(\bar x)$: 
\begin{equation}\label{eq:Ball4r}
({\rm e}_{t})_{\sharp}\nu (B_{4\bar r}(\bar x)) =1,
\end{equation}
for all $t \in [0,1]$; in particular, to obtain regularity properties of the transport set $\T$ (that in turn are deduced from regularity properties of $W_{2}$-geodesics) 
will be sufficient to assume curvature bounds inside $B_{4\bar r}(\bar x)$.
We have therefore the following 
\begin{proposition}
Let $(X,\sfd,\mm)$ be an essentially non-branching and geodesic metric measure space. Assume moreover it verifies 
$\CD_{loc}(K,N)$ in $B_{4\bar r}(\bar x)$ for some $K,N \in \R$ and $N > 1$, and that $B_{4\bar r}(\bar x)$ is relatively compact.
 
Then for any $0<\delta < \bar{r}/10$, any $\mu_{0},\mu_{1} \in \mathcal{P}(X)$ with $\mu_{0},\mu_{1}$ verifying \eqref{E:measures} and with $\mu_{0}$ absolutely continuous with respect to $\mm$, 
there exists a unique $\nu \in \Opt(\mu_{0},\mu_{1})$ and it is induced by a map; moreover there exists a Borel subset
$G \subset \Geo(X)$ such that $\nu(G) = 1$ and for each $t \in (0,1)$ the map ${\rm e}_{t}\llcorner_{G}$ is injective.
\end{proposition}

\noindent
Then we deduce that  $\bar \mm(\T \setminus \T^{b}) = 0$ and therefore
$$
\bar \mm\llcorner_{ \T} = \int_{Q} \bar \mm_{q} \, \qq(dq);
$$
with for $\qq$-a.e. $q \in Q$, $\bar \mm_{q}(X \setminus X_{q}) = 0$.

The next step consists in proving regularity properties for the conditional measures $\bar \mm_{q}$;
we have already observed that from Lemma \ref{L:convexity} each $X_{q}$ is a geodesic then, for each $q$, one can naturally 
consider the one-dimensional metric measure space $(X_{q},\sfd|_{X_{q} \times X_{q}},\bar \mm_{q})$ and prove that $\CD_{\loc}(K,N)$ localizes, i.e.  
for $\qq$-a.e. $q\in Q$ the metric measure space $(X_{q}, ,\sfd|_{X_{q} \times X_{q}},\bar \mm_{q})$ verifies $\CD(K,N)$. 
As $\T^{b} \subset B_{\bar r +2\delta}(\bar x)$ and $(X,\sfd,\mm)$ verifies $\CD_{loc}(K,N)$ inside $B_{4\bar r}(\bar x)$, 
we can repeat the same argument of \cite{CM} where the claim was proved for a generic transport set and the whole space was assumed to satisfy $\CD_{loc}(K,N)$.
For all the details we refer to \cite{CM}; here we mention that combining \cite[Theorem 3.8, Theorem 4.2, Theorem 5.1]{CM} we get the following localization result. 
Before stating it, recall that $q \mapsto \bar{\mm}_{q}$ is a $\CD(K-\ve,N)$ disintegration of $\bar{\mm}$ if and only if (by definition)  for $\qq$-a.e. $q \in Q$ the one dimensional metric measure space 
$(X_{q}, \sfd|_{X_{q} \times X_{q}},\bar \mm_{q})$ verifies $\CD(K-\ve,N)$.

\begin{theorem}\label{T:localization}
Let $(X,\sfd,\mm)$ be an essentially non-branching geodesic space.  Assume there exists  $\bar x\in X$ such that  $B_{4\bar r}(\bar x)$  is relatively compact 
and  $(X,\sfd,\mm)$ verifies $\CD_{loc}(K-\ve,N)$  inside $B_{4\bar r}(\bar x)$ for some $K,N \in \R, \ve\in [0,1]$, and $N > 1$. 

Then for $0 <\delta < \bar r/10$ and any Borel set $E \subset B_{\delta}(\bar x)$ with $\mm(E) > 0$,  the ball $B_{\bar{r}+2\delta}(\bar{x})$ can be written as the disjoint union of two sets $Z$ and $\mathcal{T}$ with $\mathcal{T}$ admitting a partition 
$\{ X_{q} \}_{q \in Q}$ and a corresponding disintegration of $\bar{\mm}\llcorner_{\mathcal{T}}$, $\{\bar{\mm}_{q} \}_{q \in Q}$ such that: 

\begin{itemize}
\item For any $\mm$-measurable set $B \subset X$ it holds 
$$
\bar{\mm}\llcorner_{\T}(B) = \int_{Q} \bar{\mm}_{q}(B \cap \T) \, \qq(dq), 
$$
where $\qq$ is a Borel measure over the $\mm$-measurable quotient set $Q \subset X$.
\medskip
\item For $\qq$-almost every $q \in Q$, the set $X_{q}$ is isometric to an interval and $\bar{\mm}_{q}$ is supported on it. 
Moreover $q \mapsto \bar{\mm}_{q}$ is a $\CD(K-\ve,N)$ disintegration.
\medskip
\item For $\qq$-almost every $q \in Q$, it holds $\int_{X_{q}} f_{E} \, \bar{\mm}_{q} = 0$ and $f_{E} = 0$ $\bar{\mm}$-a.e. in $Z$, where $f_{E}$ was defined in \eqref{eq:deffA}.
\end{itemize}
\end{theorem}

Let us mention that Theorem \ref{T:localization} holds more generally  replacing  $f_{E}$ with any $f \in L^{1}(X,\mm)$  with  $\int_{X} f \mm=0$ and $\int_{X} f(x) \, \sfd(x,\bar{x}) \, \mm(dx)<\infty$ (see \cite{CM}), but for the aims of  the present paper the above statement will suffice.

Since $(X_{q},\sfd|_{X_{q}\times X_{q}},\bar \mm_{q})$ is isomorphic (via $\f$) as metric measure space to $(I_{q}, |\cdot|, \tilde \mm_{q})$ where 
$(a_{q},b_{q})=I_{q}\subset \R$ is an interval, $|\cdot|$ is the Euclidean distance and $\tilde \mm_{q}=\varphi_{\sharp} \bar{\mm}_{q} \in \mathcal{P}(I_{q})$, 
the previous curvature property is equivalent to state that $\tilde{\mm}_{q}$  enjoys the  representation $\tilde \mm_{q} = h_{q} \L^{1}\llcorner_{I_{q}}$,
where $\L^{1}$ denotes the one dimensional Lebesgue measure, and for all $x_0,x_1 \in I_{q}$ and $t \in [0,1]$ it holds:
\begin{equation}\label{E:onedimdensity}
 h_{q}( (1-t) x_0 + t x_1)^{\frac{1}{N-1}} \geq  \sigma^{(1-t)}_{K-\ve,N-1}(|x_1-x_0|) h_{q}(x_0)^{\frac{1}{N-1}} + \sigma^{(t)}_{K-\ve,N-1}(|x_1-x_0|) h_{q}(x_1)^{\frac{1}{N-1}}  ,
\end{equation}
(recall the coefficients $\sigma_{K,N}^{t}(\theta)$ from \eqref{E:sigma}). 

\smallskip

It is worth recalling that a one-dimensional metric measure space, that for simplicity we directly identify with $((a,b), |\cdot|, h \L^{1})$ satisfies $\CD(K-\ve,N)$ 
if and only if the density $h$ verifies
\eqref{E:onedimdensity}, which in turn corresponds to the weak formulation of the more classic differential inequality 
$$
(h^{\frac{1}{N-1}}) '' + \frac{K-\ve}{N-1} h^{\frac{1}{N-1}} \leq 0.
$$

\smallskip
Coming back to the densities $h_{q}$ of the above disintegration, as they integrate to 1, from \cite[Lemma A.8]{CMi} we deduce the following 
upper bound for the density $h_{q}$:
\begin{equation}\label{E:bound}
\sup_{x \in (a_{q},b_{q})} h_{q}(x) \leq \frac{1}{b_{q}-a_{q}} \begin{cases} N & K -\ve \geq 0  \\ (\int_0^1 (\sigma^{(t)}_{K-\ve,N-1}(b_{q}-a_{q}))^{N-1} dt)^{-1} & K -\ve < 0 \end{cases};
\end{equation}
in particular, for fixed $K$ and $N$, $h_{q}$ is uniformly bounded from above as long as $b_{q}-a_{q}$ is uniformly bounded away from $0$ (and from above if $K-\ve < 0$).
Clearly since $X_{q}$ is a geodesic contained in $B_{\bar r + 2\delta}(\bar x)$, its length is trivially bounded from above.

\subsection{Sharp Bishop-Gromov inequality in essentially non-branching $\CD_{loc}(K,N)$-spaces}\label{Ss:doubling}
The sharp Bishop-Gromov inequality in non-branching $\CD_{loc}(K,N)$-spaces was proved in \cite{cavasturm:MCP}. 
In this short section, for reader's convenience,  we give a short self-contained proof of the sharp Bishop-Gromov inequality tailored to our different framework.

As  always in this section we assume $(X,\sfd,\mm)$ to be essentially non-branching and to  satisfy the $\CD_{loc}(K-\ve,N)$ condition inside  $B_{4 \bar r}(\bar{x})$.
Theorem \ref{T:localization} permits to localize any function $f_{E}$ defined in \eqref{eq:deffA}, with $E \subset B_{\delta}(\bar x)$.
As one can notice from the construction we presented in this section, the only place where the definition of $f_{E}$ plays a role is in the 
definition of the optimal transport problem from $\mu_{0} = f^{+}_{E} \bar \mm$ to $\mu_{1} = f^{-}_{E} \bar \mm$; from there, one considers 
the associated 1-Lipschitz Kantorovich potential $\f : X \to \R$ and builds the localization for $f_{E}$. 

The same procedure can be repeated verbatim replacing $\varphi$ with any $1$-Lipschitz function inside $B_{4\bar r}(\bar{x})$; in particular we are interested in considering $\f_{\bar x} : = \sfd(\cdot,\bar x)$.
As the space is geodesic, $\bar \mm( B_{4 \bar r}(\bar x) \setminus \T ) = 0$ and the transport set $\T$ is formed by geodesics starting from $\bar x$. In particular this permits 
to prove that for any $0 \leq r < 4 \bar r$ 
\begin{equation}\label{eq:disintbarmmBG}
\bar \mm (B_{r}(\bar x) ) = \int_{Q}   \tilde{\mm}_{q}([0,r]) \, \qq(dq),
\end{equation}
where $q \to \bar \mm_{q}$ is a $\CD(K-\ve,N)$ disintegration and  $\tilde \mm_{q}=\varphi_{\sharp} \bar{\mm}_{q} \in \mathcal{P}(I_{q})$ was defined in the previous subsection.
From the classical Bishop-Gromov volume growth inequality, we know that for any $0 < r < R < 4 \bar r$

\begin{equation}\label{eq:BG1D}
\tilde \mm_{q}([0,r]) \geq \tilde \mm_{q}([0,R]) \frac{V_{K-\ve,N}(r) }{V_{K-\ve,N}(R)}.
\end{equation}
The combination of \eqref{eq:disintbarmmBG} and \eqref{eq:BG1D}  implies the sharp Bishop-Gromov volume growth inequality: for any $0 < r < R < 4 \bar r$
\begin{equation}\label{E:bishop}
\frac{\bar \mm(B_{r}(\bar x))}{\bar \mm(B_{R}(\bar x))}\geq \frac{V_{K-\ve,N}(r)}{V_{K-\ve,N}(R)}.
\end{equation}
\noindent
For later, it will also be convenient to estimate the volume of annuli. As above, for any $0 \leq r_{1} < r_{2} < 4 \bar r$ 
%
\begin{equation}\label{eq:disintbarmmBGAnnuli}
\bar \mm (B_{r_{2}}(\bar x)\setminus B_{r_{1}}(\bar x) ) = \int_{Q}\tilde  \mm_{q}([r_{1},r_{2}]) \, \qq(dq),
\end{equation}
where $q \to \bar \mm_{q}$ is a $\CD(K-\ve,N)$ disintegration. From the classical Bishop-Gromov volume growth inequality, we know that for any $0 < r_{1} < r_{2}  < 4 \bar r$
\begin{equation}\label{eq:BG1DAnnuli}
\tilde \mm_{q}([r_{1},r_{2}]) \leq  \frac{  \tilde \mm_{q}([0,r_{1}]) }{V_{K-\ve,N}(r_{1})}  \; \Big(V_{K-\ve,N}(r_{2})- V_{K-\ve,N}(r_{1}) \Big) .
\end{equation}
The combination of \eqref{eq:disintbarmmBGAnnuli} and \eqref{eq:BG1DAnnuli}  implies the following estimate on the measure of annuli
\begin{equation}\label{E:mmAnnuli}
\mm(B_{r_{2}}(\bar{x})\setminus B_{r_{1}}(\bar{x}))\leq  \frac{\mm(B_{r_{1}}(\bar{x})) }{V_{K-\ve,N}(r_{1})}  \; \Big(V_{K-\ve,N}(r_{2})- V_{K-\ve,N}(r_{1}) \Big), \quad \forall 0<r_{1}<r_{2}< 4 \bar{r}.
\end{equation}


\section{Proof of the main results}\label{S:ProofMain}
First of all note that  we can assume  $K$ and $N$ to be such that $\sup_{r>0}{\rm Vol}_{K,N}(r)> 1$, so that  $\bar{r}_{K,N} < \pi \sqrt{ (N-1)/K}$. Indeed, if  $\bar{r}_{K,N}= \pi \sqrt{ (N-1)/K}$ then $B_{\bar{r}_{K,N}}(\bar{x})=X$ and Theorem \ref{thm:MainNonSmooth} follows directly by the  global Levy-Gromov Inequality that was established for essentially non-branching $\CD_{loc}(K,N)$ spaces in \cite{CM}.
 
Thus from now on we will work under the assumption that  ${\rm Vol}_{K,N}(\bar{r})=1$,   $\mm(B_{\bar{r}}(\bar{x})) \geq 1-\eta$ and by the very definition \eqref{eq:defbarmm}  of $\bar{\mm}$ we have $\bar{\mm}(B_{\bar{r}+2\delta}(\bar{x}))=1$.
\\In the last section, fixed any Borel subset $E \subset B_{\delta}(\bar x)$, we considered the function  $f_{E}$ defined in  \eqref{eq:deffA} and in Theorem  \ref{T:localization} we constructed a 1-dimensional localization. 
\\The last assertion of Theorem \ref{T:localization} implies that
\begin{align*}
0& = \int_{X_{q}} f_{E} \, \bar{\mm}_{q}=  \bar{\mm}_{q}(E\cap X_{q}) - \frac{\bar{\mm}(E)}{\bar{\mm}(B_{\bar{r}}(\bar{x}))}  \bar{\mm}_{q} 
(B_{\bar{r}}(\bar{x})), \quad \text{for $\qq$-a.e. $q \in Q$};
  \end{align*}
thus
  \begin{equation}\label{eq:mmqAmmA}
  \bar{\mm}_{q}(E\cap X_{q})=  \frac{ \bar{\mm}_{q} (B_{\bar{r}}(\bar{x}))}{\bar{\mm}(B_{\bar{r}}(\bar{x}))} \, \bar{\mm}(E).
  \end{equation}
Moreover, from $B_{4\bar r}(\bar x)\subset \supp(\mm)$, we have 
$\bar \mm(E)  < \bar \mm(B_{\bar r}(\bar x))$ and therefore
$f_{E} >0$ over $E$ and $f < 0$ over $B_{\bar r}(\bar x) \setminus E$;
then since  $f_{E} = 0$ $\bar{\mm}$-a.e. in $Z$, necessarily  $\bar{\mm}( B_{\bar{r}}(\bar{x}) \cap Z)=0$, i.e. 
\begin{equation}\label{eq:BT}
\bar{\mm}( B_{\bar{r}}(\bar{x}) \setminus  {\mathcal T})=0.
\end{equation}
In particular, for any $0<\rho\leq \delta \leq \bar{r}_{K,N}/10$ it holds
$$
\bar{\mm}(E^{\rho})=\bar{\mm}(E^{\rho}\cap  {\mathcal T})=\int_{Q} \bar{\mm}_{q} (E^{\rho}) \, \qq(dq). 
$$
Therefore
\begin{align}
\bar{\mm}^{+}(E)&= \liminf_{\rho \downarrow 0} \frac{\bar{\mm}(E^{\rho})- \bar{\mm}(E)}{\rho}=   \liminf_{\rho \downarrow 0} \frac{\bar{\mm}(E^{\rho} \cap {\mathcal T})- \bar{\mm}(E \cap {\mathcal T} )}{\rho} \nonumber \\
&=  \liminf_{\rho \downarrow 0}  \int_{Q}\frac{\bar{\mm}_{q}(E^{\rho} \cap X_{q})- \bar{\mm}_{q}(E \cap X_{q} )}{\rho} \, \qq(dq).
\end{align}
Calling now $E_{q}:= E \cap X_{q}$ and noting that $E_{q}^{\rho}\cap X_{q} \subset E^{\rho} \cap X_{q}$, 
we can continue the chain of inequalities by using Fatou's Lemma  as
\begin{align}
\bar{\mm}^{+}(E)&\geq   \liminf_{\rho \downarrow 0}  \int_{Q}\frac{\bar{\mm}_{q}(E^{\rho}_{q})- \bar{\mm}_{q}(E_{q})}{\rho} \, \qq(dq) 
\geq  \int_{Q}   \liminf_{\rho \downarrow 0} \frac{\bar{\mm}_{q}(E^{\rho}_{q})- \bar{\mm}_{q}(E_{q})}{\rho} \,  \qq(dq) \nonumber \\
&=  \int_{Q}  \bar{\mm}^{+}_{q}(E_{q}) \, \qq(dq) \geq   \int_{Q} \cI_{K-\varepsilon, N, \bar{r}+\delta} (\bar{\mm}_{q}(E_{q})) \,  \qq(dq)  \nonumber \\
&=   \int_{Q} \cI_{K-\varepsilon, N, \bar{r}+ \delta} \left( \frac{ \bar{\mm}_{q} (B_{\bar{r}}(\bar{x}))}{\bar{\mm}(B_{\bar{r}}(\bar{x}))} \, \bar{\mm}(E) \right) \,  \qq(dq), 
  \label{eq:MCmmq}
\end{align}
where we used \cite[Theorem 6.3]{CM} and that $\textrm{Length}(X_{q}) \leq  \bar{r}+\delta$ by Lemma \ref{lem:2delta} to infer  
$\bar{\mm}^{+}_{q}(E_{q})\geq  \cI_{K-\varepsilon, N, \bar{r}+\delta} (\bar{\mm}_{q}(E_{q}))$ for $\qq$-a.e. $q \in Q$
and where we recalled \eqref{eq:mmqAmmA} to get the last identity.
\\

We next claim that
\begin{equation}\label{eq:claimmmq}
\left| \frac{ \bar{\mm}_{q} (B_{\bar{r}}(\bar{x}))}{\bar{\mm}(B_{\bar{r}}(\bar{x}))} \, \bar{\mm}(E) -  \bar{\mm}(E)  \right| \leq C_{K,N}  (\delta+\varepsilon) \;  \bar{\mm}(E)   , 
\quad \forall \delta \in (0, \bar{\delta}_{K,N}], \; \forall \varepsilon\in (0,\bar{\varepsilon}_{K,N}],
\end{equation}
where $C_{K,N}> 0, \bar{\delta}_{K,N} \in (0, \bar{r}/10], \bar{\varepsilon}_{K,N}\in (0,1]$ depend just on $K$ and $N$.
\\To show the claim observe first that by Bishop-Gromov volume comparison \eqref{E:bishop}, combined with $\textrm{Vol}_{K-\varepsilon,N}(\bar{r}_{K-\varepsilon,N})=1=\bar{\mm}(B_{\bar{r} + 2\delta}(\bar{x}))$ and with $\bar{r}_{K-\varepsilon,N} \leq \bar{r}:=\bar{r}_{K,N}\leq  \bar{r}_{K-\varepsilon,N}+C_{K,N} \varepsilon$, it follows
\begin{align} 
\bar{\mm}(B_{\bar{r}}(\bar{x})) &\geq  \frac{\textrm{Vol}_{K-\varepsilon,N}(\bar{r})} {\textrm{Vol}_{K-\varepsilon,N}(\bar{r}+2 \delta)} \;  \bar{\mm}(B_{\bar{r} + 2\delta}(\bar{x})) \geq  \frac{1} {\textrm{Vol}_{K-\varepsilon,N}( \bar{r}_{K-\varepsilon,N}+C_{K,N} \varepsilon+2 \delta)} \nonumber \\
& \geq 1- C_{K,N} (\delta+\varepsilon),\quad \forall  \delta \in (0, \bar{\delta}_{K,N}], \; \forall \varepsilon\in (0,\bar{\varepsilon}_{K,N}], \label{eq:BGmmq}
  \end{align}
for some $C_{K,N}> 0, \bar{\delta}_{K,N} \in (0, \bar{r}/10]$, $ \bar{\varepsilon}_{K,N}\in (0,1]$ depending just on $K$ and $N$. 
\\On the other hand, recalling that $\bar{\mm}_{q}(B_{\bar{r}+2 \delta}(\bar{x}))=1$, we have
\begin{align}
\bar{\mm}_{q}(B_{\bar{r}}(\bar{x}))&= \bar{\mm}_{q}(B_{\bar{r}+2 \delta}(\bar{x}))- \bar{\mm}_{q}(B_{\bar{r}+2 \delta}(\bar{x}) \setminus B_{\bar{r}}(\bar{x})) \nonumber \\
&= 1-  \bar{\mm}_{q}(B_{\bar{r}+2 \delta}(\bar{x}) \setminus B_{\bar{r}}(\bar{x})). \label{eq:mqAnnulus}
\end{align}
Notice that for all $x \in B_{\delta}(\bar{x})$  and $y \in B_{\bar{r}+2 \delta}(\bar{x}) \setminus B_{\bar{r}}(\bar{x})$, by triangle inequality we have
$$\sfd(x, y)\geq \sfd(\bar{x}, y)- \sfd(\bar{x},x) \geq \bar{r}-\delta, $$
thus, recalling that $\textrm{Length}(X_{q}) \leq \bar{r}+\delta$, we get
\begin{equation}\label{eq:XqBdelta}
X_{q}^{-1} \big(B_{\bar{r}+2 \delta}(\bar{x}) \setminus B_{\bar{r}}(\bar{x}) \big) \subset \left[ \frac{\bar{r}-\delta} {\bar{r}+\delta}, 1  \right],
\end{equation}
where, with a slight abuse of notation, in the last formula we denoted with $X_{q}:[0,1]\to X$ the constant-speed parametrization of the ray $X_{q}$.
\\Moreover if $X_{q}^{-1} \big(B_{\bar{r}+2 \delta}(\bar{x}) \setminus B_{\bar{r}}(\bar{x}) \big)\neq \emptyset$, necessarily (as each $X_{q}$ intersects $B_{\delta}(\bar x)$)
the length of $X_{q}$ is at least $\bar r/2$;
then since $\bar{\mm}_{q}$ is a $\CD(K-\varepsilon, N)$ disintegration, we can invoke \eqref{E:bound} and obtain from \eqref{eq:XqBdelta} that
\begin{equation}\label{eq:mmqAnnulusleq}
\bar{\mm}_{q}  \big( B_{\bar{r}+2 \delta}(\bar{x}) \setminus B_{\bar{r}}(\bar{x}) \big) \leq C_{K,N} \delta, \quad \forall \delta \in (0,\bar{\delta}_{K,N}].
\end{equation}
The combination of \eqref{eq:BGmmq},   \eqref{eq:mqAnnulus}  and \eqref{eq:mmqAnnulusleq}   gives the desired claim \eqref{eq:claimmmq}.
\\

We can now conclude the proof. 
Since $E \subset B_{\delta}(\bar{x})$ for $\delta\in (0, \bar{\delta}_{K,N}]$ small and by assumption (2) $\mm(B_{\delta}(\bar{x}))$ is small,
we are just interested on the model isoperimetric profile function $\cI_{K,N,D}(v)$ for small $v>0$. 
Recall from \eqref{eq:IKNvN-1Nprel}-\eqref{eq:limsupoKNprel} that 
\begin{equation}\label{eq:IKNvN-1N}
\cI_{K,N,\bar{r}} (v)  = N \omega_{N}^{1/N}  \; v^{\frac{N-1}{N}}+O_{K,N}(v^{\frac{3(N-1)}{N}}), \quad \forall v \in [0, 1], 
\end{equation}
where $(0,1]\ni t\mapsto O_{K,N}(t) \in \R$ is a smooth function depending smoothly on $K,N$ and satisfying 
\begin{equation}\label{eq:limsupoKN}
\limsup_{t\downarrow 0} \frac{|O_{K,N}(t)|}{t}< \infty.
\end{equation}
By slightly perturbing the lower bound on the Ricci curvature and the upper bound on the diameter, by the smoothness of the map $\R\times \R_{>1} \times \R_{>0}\ni(K,N,D) \mapsto \cI_{K,N,D}(v)$ for fixed $v>0$ sufficiently small (to this aim see the discussion after  \eqref{eq:IKNvN-1Nprel}-\eqref{eq:limsupoKNprel}) it  is also readily  seen  that for $N\geq 2$ 
\begin{align}
\cI_{K-\varepsilon,N,\bar{r}+\delta} (v) & \geq \cI_{K-\varepsilon,N,\bar{r}_{K-\varepsilon,N}+(\bar{r}-\bar{r}_{K-\varepsilon,N})+\delta} (v)   \geq \big(1-C_{K,N} (\delta+\varepsilon)\big)\,  \cI_{K-\varepsilon,N,\bar{r}_{K-\varepsilon,N}} (v)  \nonumber \\
&   \overset{\eqref{eq:IKNvN-1N}}{=}   \big(1-C_{K,N} (\delta+\varepsilon)\big)\, \Big[ N \omega_{N}^{1/N}  \; v^{\frac{N-1}{N}}+O_{K-\varepsilon,N}(v^{\frac{3(N-1)}{N}}) \Big]  \nonumber \\
& \; =   N \omega_{N}^{1/N}  \big(1-C_{K,N} (\delta+\varepsilon)+ O_{K-\varepsilon,N}(v^{\frac{2(N-1)}{N}})\big)  \; v^{\frac{N-1}{N}}    \nonumber \\
& \; \geq  N \omega_{N}^{1/N}  \; \big(1-C_{K,N}  (\delta+\varepsilon) \big)  \; v^{\frac{N-1}{N}}, \quad \forall v \in [0, C_{K,N} \delta].  \label{eq:IKNvN-1Neps}
\end{align}
Since by Bishop-Gromov inequality and assumption (2) we have
\begin{align*}
\bar{\mm}(E)&\leq \bar{\mm}(B_{\delta}(\bar{x})) \leq \left( \limsup_{r\downarrow 0} \frac{\bar{\mm} (B_{r}(\bar{x}))} {\textrm{Vol}_{K-\varepsilon,N}(r)} \right)  \textrm{Vol}_{K-\varepsilon,N}(\delta) \leq (1+\eta) \;  \textrm{Vol}_{K-\varepsilon,N}(\delta), 
\end{align*}
we can apply  \eqref{eq:claimmmq} and  \eqref{eq:IKNvN-1Neps}  to get 
\begin{align}
\cI_{K-\varepsilon,N,\bar{r}+\delta} \left(  \frac{ \bar{\mm}_{q} (B_{\bar{r}}(\bar{x}))}{\bar{\mm}(B_{\bar{r}}(\bar{x}))} \, \bar{\mm}(E) \right) & \overset{\eqref{eq:IKNvN-1Neps}} {\geq} N \omega_{N}^{1/N}  \; 
\big(1-C_{K,N} (\delta+\varepsilon) \big)   \left(  \frac{ \bar{\mm}_{q} (B_{\bar{r}}(\bar{x}))}{\bar{\mm}(B_{\bar{r}}(\bar{x}))} \, \bar{\mm}(E) \right)^{\frac{N-1}{N}} \nonumber\\
& \overset{\eqref{eq:claimmmq}}  {\geq}   N \omega_{N}^{1/N}  \; \big(1-C_{K,N} (\delta+\varepsilon) \big)  \; \big( 1-C_{K,N} (\delta+\varepsilon) \big)^{\frac{N-1}{N}}\;  \bar{\mm}(E)^{\frac{N-1}{N}}  \nonumber\\
&\; \geq   N \omega_{N}^{1/N}  \; \big(1-C_{K,N} (\delta+\varepsilon) \big)  \;   \bar{\mm}(E)^{\frac{N-1}{N}}.  \label{eq:CLaimfinal}
\end{align}
Recalling that $\qq(Q)=1$,  putting together \eqref{eq:MCmmq} and \eqref{eq:CLaimfinal}   yields
$$
\bar{\mm}^{+}(E)\geq   N \omega_{N}^{1/N}  \; \big(1-C_{K,N} (\delta+\varepsilon) \big)  \;   \bar{\mm}(E)^{\frac{N-1}{N}},
$$
which, combined with \eqref{eq:AssB} and  \eqref{eq:defbarmm}  gives the desired almost euclidean isoperimetric inequality
\begin{align}
\mm^{+}(E)&\geq  \mm(B_{\bar{r} + 2 \delta}(\bar{x}))^{1/N}  N \omega_{N}^{1/N}  \; \big(1-C_{K,N} (\delta+\varepsilon) \big)  \;   \mm(E)^{\frac{N-1}{N}},  \nonumber\\
&\geq N \omega_{N}^{1/N}  \; \big(1-C_{K,N} (\delta+\varepsilon+\eta) \big)  \;   \mm(E)^{\frac{N-1}{N}}. \nonumber
\end{align}

\end{document}